\documentclass[12pt]{article}
\usepackage{fullpage}
\usepackage{times}
\usepackage{amsmath}
\usepackage{latexsym}
\usepackage{amssymb,amsfonts,amsthm}
\usepackage{graphicx}
\usepackage{graphics}
\usepackage{float}
\usepackage{multirow}
\usepackage{multicol}
\usepackage{rotating}
\usepackage{setspace}
\usepackage{mathrsfs}
\usepackage{tikz}
\usepackage{tabularx}
\usetikzlibrary {graphs.standard}
\usetikzlibrary {graphs}
\usetikzlibrary{positioning,chains,fit,shapes,calc}

 \renewcommand{\-}{\ensuremath{\text{--}}}

\let\oldmarginpar\marginpar
\renewcommand\marginpar[1]{\-\oldmarginpar[\raggedleft\footnotesize #1]%
{\raggedright\footnotesize #1}}

\newtheorem{theorem}{Theorem}
\newtheorem{lemma}[theorem]{Lemma}

\renewenvironment{proof}{\noindent{\it Proof}:}{\hfill $\blacksquare$}

\newtheorem{corollary}[theorem]{Corollary}
\DeclareMathOperator{\lcm}{lcm}

\begin{document}

\title{On the Hamilton-Waterloo Problem with a single factor of 6-cycles}

\author{Zazil Santizo Huerta and Melissa Keranen\\
Michigan Technological University\\
Department of Mathematical Sciences\\
Houghton, MI 49931, USA\\
}

\maketitle

 \begin{abstract}
The uniform Hamilton-Waterloo Problem (HWP) asks for a resolvable $(C_M, C_N)$-decomp-\\osition of $K_v$ into $\alpha$ $C_M$-factors and $\beta$ $C_N$-factors. We denote a solution to the uniform Hamilton Hamilton-Waterloo problem by $\hbox{HWP}(v; M, N; \alpha, \beta)$. Our research concentrates on addressing some of the remaining unresolved cases, which pose a significant challenge to generalize. We place a particular emphasis on instances where the $\gcd(M,N)=\{2, 3\}$,  with a specific focus on the parameter $M=6$. We introduce  modifications to some known structures, and develop new approaches to resolving these outstanding challenges in the construction of uniform $2$-factorizations. This innovative method not only extends the scope of solved cases, but also contributes to a deeper understanding of the complexity involved in solving the Hamilton-Waterloo Problem.

 \end{abstract}

\section{Introduction}
 
The Oberwolfach problem, proposed by Ringel in 1967, was posed during a graph theory meeting held in the German city of the same name. The goal of this problem is to seat 
$v$ conference attendees at $t$ round tables over $\frac{v-1}{2}$ 
nights such that each attendee sits next to each other attendee exactly once. 
It is clear that $v$ must be odd in order for $\frac{v-1}{2}$ to be an integer, and the total 
number of seats around the tables must be $v$. A variant of this problem, called the 
spouse avoiding variant, considers an even number of attendees over $\frac{v-2}{2}$ nights. 
In this situation, each attendee sits next to each other attendee, except their spouse, exactly once.

\vspace{5mm}

One of the most studied extensions of the Oberwolfach problem, is the Hamilton-Waterloo problem. This problem uses the same idea, but the conference is held at two different venues, Hamilton and Waterloo. The attendees spend $\alpha$ nights in Hamilton, with a particular table arrangement, and $\beta$ nights in Waterloo, 
with a different table arrangement. It is not necessary that the tables at a particular venue all have the same size, but if they do, then we say it is the uniform case of the problem.

\

If we translate the uniform case of this problem to graph theory, then a solution is equivalent to finding a decomposition of the complete graph $K_v$ (or $K_v-F$ when $v$ is even) into $\alpha$
2-factors where each 2-factor consists entirely of $M$-cycles (a $C_{M}$-factor), and $\beta$ 2-factors where each 2-factor consists entirely of $N$-cycles (a $C_{N}$-factor). We will denote the Hamilton-Waterloo Problem by $\hbox{HWP}(v;M,N;\alpha,\beta)$ and will refer to a solution of $\hbox{HWP}(v;M,N;\alpha,\beta)$ as a resolvable $(C_{M},C_{N})$-decomposition of $K_{v}$ into $\alpha$ $C_{M}$-factors and $\beta$ $C_{N}$-factors. When $\alpha=0$ or $\beta=0$, then it is actually the uniform Oberwolfach problem, which has been solved.

\begin{theorem}\label{fac}
\cite{AL, ALST}
    Let $v,M \geq 3$ be integers. There is a $C_M$-factorization of $K_v$ (or ${K_v}-F$ when $v$ is even) if and only if $M|v$; except that there is no $C_3$-factorization of ${K_6}-F$ or ${K_{12}}-F$.
\end{theorem}

\

The following theorem summarizes the necessary conditions for the existence of a solution to $\hbox{HWP}(v;M,N;\alpha,\beta)$.

\begin{theorem}
\label{Ncond}
\cite{ABBE} In order for a solution of  $\hbox{HWP}(v;M,N;\alpha,\beta)$ to exist, we must have:
 \begin{enumerate}
  \item If $v$ is odd, $\alpha+\beta=\frac{v-1}{2}$,
  \item If $v$ is even, $\alpha+\beta=\frac{v-2}{2}$,
  \item If $\alpha>0$, $M|v$,
  \item If $\beta>0$, $N|v$.
 \end{enumerate}
\end{theorem}

\
Most results on this problem have been obtained by focusing on the parity conditions of $M$ and $N$, the divisibility conditions of $M$ and $N$, or by fixing $M$ or $N$. Recently, Burgess, et.al. have made progress and have given the following result, which summarizes the current state of the problem.

\begin{theorem}
    \cite{BUR}
    \label{Burgess}
     Let $v$, $M$ and $N$ be integers greater than 3, and let $l =
lcm(M, N)$. A solution to $\hbox{HWP}(v; M, N; \alpha, \beta)$ exists if and only if $M | v$
and $N | v$, except possibly when
\begin{itemize}
    \item $\gcd(M, N) \in\{1, 2\}$;
    \item 4 does not divide $v/l$;
    \item $v/4l \in \{1, 2\}$;
    \item $v = 16l$ and $\gcd(M, N)$ is odd;
    \item $v = 24l$ and $\gcd(M, N) = 3$.
\end{itemize}
\end{theorem}

As we get closer to obtaining a complete solution to this problem, the standard techniques that have been used to date do not apply in straightforward way. Therefore, new approaches need to be developed. The techniques used in this manuscript were developed for $M=6$ with the hope of future generalization to all $M$, which would further close the gap on the open cases for this problem.

In section 4, we introduce the modified row sum matrix which produces a significant advancement to easily keeping track of the edges used in our decompositions. It also allow us to develop techniques for combining the inside edges and cross differences in our decompositions when considering the complete graph as the union of complete graphs and an equipartite graph.

In section 6, we employ specific techniques in a unique way. We consider  the Walecki decomposition of a complete graph, then make 4-cycle switches and 6-cycle switches to find desired decompositions when $\gcd(M,N)= 2$. The technique of weighting is used heavily in the case when $gcd(6,N)= 3$.

\
There are different strategies for tackling the open cases for the Hamilton-Waterloo problem at this time. One strategy is to focus on the notoriously difficult case when there is a single factor of a specific type. Thus, our aim in this work is to find solutions when $\alpha=1$. We restrict our single factor to be a uniform factor of 6-cycles. Thus, the problem is to find solutions to $\hbox{HWP}(v;6,N;1,\beta)$. With a focus on $M=6$, we solve some many of the open cases listed in Theorem \ref{Burgess}. In particular, we solve the problem when $\gcd(6,N)=2$ and $\alpha=1$; and when $\gcd(6,N)=3$, $v=24l$, and $\alpha=1$, except for the case when $N=2x$ and $x=3$. Our main Theorem, listed here, is proved at the end of section 3. 

\begin{theorem}
\label{Main}
    Let $x$, $t$, and $p$ be integers such that $t \geq 1$, $x \geq 2$ and $x \neq 6p$. A solution to $\hbox{HWP}(v;6,N;1,\beta)$, where $v=6xt$, $N=2x$, and $\beta=3xt-2$, exists if and only if $6|v$ and $N|v$, except possibly when $\gcd(6,N)=1$, or $\gcd(6,N)=2$ and $x=3$.
\end{theorem}

\section{Preliminary Results}

If $G$ is a graph with $n$ vertices, the difference of an edge $\{u,v\}$ in $G$ with $u <v $ is defined to be $v-u$ or $n-(v-u)$, whichever is smaller. For any subset $D\subseteq \{1,2, \dots, \lfloor g/2 \rfloor\}$, $G(D,g)$ is defined to be the graph with vertex set
$$V(G(D,g))=\{0,1,\dots,g-1\}$$
and edge set consisting of all edges having a difference in $D$, that is, the edge set is defined as
$$E(G(D,g))=\{\{u,v\}| D(u,v)\in D\}$$
where $D(u,v)= \hbox{min}\{v-u,g-(v-u)\}$ and $u<v$. Then, we have the following lemma.

\begin{lemma}
\label{Stern-Lenz}
\cite{LR}
If $h$ is the $\gcd$ of $D\cup \{g\}$, then $G(D,g)$ consists of $h$ components, each of which is isomorphic to $G(\{\frac{d}{h} |d\in D\},\frac{g}{h})$.
\end{lemma}

\begin{corollary} 
\label{Stern-Lenz2}
\cite{LR}
$G(\{d\},g)$ consists of $h=\gcd(\{d,g\})$ components, and each component is
\begin{enumerate}
\item{a $(\frac{g}{h})$-cycle if $d\neq \frac{g}{2}$}, and
\item{$K_2$ if $d=\frac{g}{2}$}.
\end{enumerate}
\end{corollary}

\

An \emph{equipartite graph} is a graph whose vertex set can be partitioned into $u$ subsets of size $h$ such that no two vertices from the same subset are connected by an edge.  The complete equipartite graph with $u$ subsets of size $h$ is denoted $K_{(h:u)}$, and it contains every edge between vertices of different subsets. The following result solves the Oberwolfach Problem for constant cycle lengths over complete equipartite graphs (as opposed to $K_v$).  That is to say, with finitely many exceptions, $K_{(h:u)}$ has a resolvable $C_m$-factorization.

\

\begin{theorem}
\label{equip}
\cite{L}  For $m\geq 3$ and $u\geq 2$, $K_{(h:u)}$ has a resolvable $C_m$-factorization if and only if $hu$ is divisible by $m$, $h(u-1)$ is even, $m$ is even if $u=2$, and $(h,u,m) \not \in \{(2,3,3), (6,3,3), \\ (2,6,3),(6,2,6)\}$.
\end{theorem}

\

Let $G_{0}, G_{1}, \ldots G_{x-1}$ be the $x$ parts of the equipartite graph $K_{(g:x)}$ with $G_{i}$ having vertex set $\{(i,j):0 \leq j \leq g-1\}$. Define a \emph{cross difference} to be the difference between two vertices from different parts of $K_{(g:x)}$. The cross difference from $(i_k,j_m)$ to $(i_{k+1},j_{m+1})$ is denoted by $d_k \equiv j_{m+1}-j_m \pmod{g}$, where $i_k \neq i_{k+1}$. Define an {\it inside difference} to be the difference between two vertices from the same part of $K_{(g:x)}$.

\

For a nonempty set $X$ of edges of a graph $G$, the subgraph $G[X]$ induced by $X$ has $X$ as its edge set and a vertex $v$ belongs to $G[X]$ if $v$ is incident with at least one edge in $X$. A subgraph $H$ of $G$ is edge-induced if there is a nonempty subset $X$ of $E(G)$ such that $H=G[X]$.

\

Let $G_{i_{k}}[D]$ be the edge-induced subgraph of $K_g$ on $G_{i_k}$, consisting of all inside edges of $K_g$ on $G_{i_k}$ having a difference in $D$. If $X$ is simply a set of edges from $K_g$, then $G_{i_k}[X]$ is the edge induced subgraph $G[X]$ on $G_{i_k}$.

\section{The Subgraph $G[\Delta]$}

In this section, we will investigate particular subgraphs of $K_{(g:x)}$. Let $G=K_{(g:x)}$ and consider a 2-regular spanning graph $H$ on $x$ vertices. Suppose that for each edge  $\{i_k,i_{k+1}\} \in E(H)$, there is a cross difference, $d_k$, assigned to that edge. Let $D_k$  be the set of edges from parts $G_{i_k}$ to $G_{i_{k+1}}$ with cross difference $d_k$. Then $G[D_k]$ is the edge induced subgraph of $G$ consisting of all edges from $G_{i_k}$ to $G_{i_{k+1}}$ with cross difference $d_k$. If $(d_1, d_2, \ldots, d_x)$ is a list of differences, then $\Delta=(D_1,D_2, \ldots, D_x)$ is a list of edges with particular cross differences. We denote $G[\Delta]$ as the edge induced subgraph of $G$ consisting of those edges. 

\

{\bf Example:} Let $G=K_{(3:7)}$, and suppose $H$ consists of a 4-cycle $(i_1,i_2,i_3,i_4)=(0,2,1,3)$ and a 3-cycle $(i_5,i_6,i_7)=(4,5,6)$. Given the list of differences $(d_1,d_2,d_3,d_4,d_5,d_6,d_7)=(0,1,1,2,1,1,1)$, then $$\Delta=(D_1,D_2,D_3,D_4,D_5,D_6,D_7)$$ 

with 

$$D_1=\{\hbox{edges from} \ G_0 \ \hbox{to} \ G_2 \ \hbox{with difference} \ 0\},$$
$$D_2=\{\hbox{edges from} \ G_2 \ \hbox{to} \ G_1 \ \hbox{with difference} \ 1\},$$
$$D_3=\{\hbox{edges from} \ G_1 \ \hbox{to} \ G_3 \ \hbox{with difference} \ 1\},$$
$$D_4=\{\hbox{edges from} \ G_3 \ \hbox{to} \ G_0 \ \hbox{with difference} \ 2\},$$
$$D_5=\{\hbox{edges from} \ G_4 \ \hbox{to} \ G_5 \ \hbox{with difference} \ 1\},$$
$$D_6=\{\hbox{edges from} \ G_5 \ \hbox{to} \ G_6 \ \hbox{with difference} \ 1\},$$
$$D_7=\{\hbox{edges from} \ G_6 \ \hbox{to} \ G_4 \ \hbox{with difference} \ 1\},$$
\bigskip

Then $G[\Delta]$ is the subgraph of $G$ consisting of these edges. We give the subgraph $G[\Delta]$ in Figure \ref{0}. Notice that each cycle, $c=(i_1,i_2, \ldots, i_k)$, gives rise to cycles in $G[\Delta]$ on the vertices in $G_{i_1} \cup G_{i_2} \cup \ldots \cup G_{i_k}$. The next result characterizes the cycles in $G[\Delta]$ when $H$ is Hamilton cycle.

\

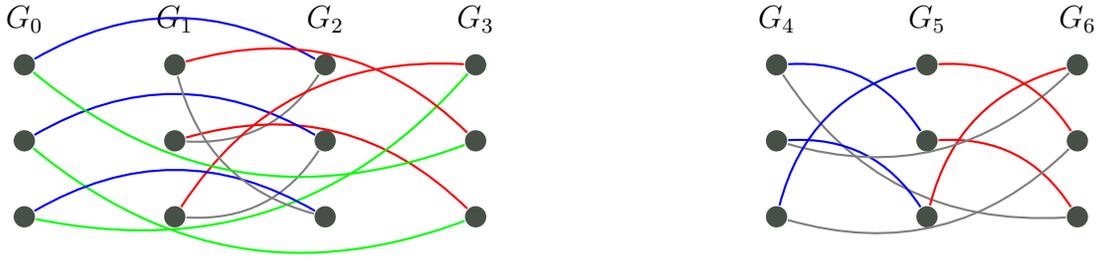
\begin{figure}[h!]
\centering
\definecolor{mygray}{RGB}{70,80,70}

\begin{tikzpicture}[thick,
  every node/.style={circle, inner sep=0.1cm},
  fsnode/.style={fill=mygray},
  gsnode/.style={fill=mygray},
  hsnode/.style={fill=mygray},
  isnode/.style={fill=mygray}
]

\begin{scope}[start chain=going below,node distance=7mm]
  \node[fsnode,on chain] (f1) [label={above:$G_0$}]{};
  \node[fsnode,on chain] (f2) {};
  \node[fsnode,on chain] (f3) {};
\end{scope}

\begin{scope}[xshift=2cm,yshift=0cm,start chain=going below,node distance=7mm]
  \node[gsnode,on chain] (g1) [label={above:$G_1$}]{};
  \node[gsnode,on chain] (g2) {};
  \node[gsnode,on chain] (g3) {};
\end{scope}

\begin{scope}[xshift=4cm,yshift=0cm,start chain=going below,node distance=7mm]
  \node[hsnode,on chain] (h1) [label={above:$G_2$}]{};
  \node[hsnode,on chain] (h2) {};
  \node[hsnode,on chain] (h3) {};
\end{scope}

\begin{scope}[xshift=6cm,yshift=0cm,start chain=going below,node distance=7mm]
  \node[isnode,on chain] (i1) [label={above:$G_3$}]{};
  \node[isnode,on chain] (i2) {};
  \node[isnode,on chain] (i3) {};
\end{scope}

\begin{scope}[xshift=10cm,yshift=0cm,start chain=going below,node distance=7mm]
  \node[isnode,on chain] (j1) [label={above:$G_4$}]{};
  \node[isnode,on chain] (j2) {};
  \node[isnode,on chain] (j3) {};
\end{scope}

\begin{scope}[xshift=12cm,yshift=0cm,start chain=going below,node distance=7mm]
  \node[isnode,on chain] (k1) [label={above:$G_5$}]{};
  \node[isnode,on chain] (k2) {};
  \node[isnode,on chain] (k3) {};
\end{scope}

\begin{scope}[xshift=14cm,yshift=0cm,start chain=going below,node distance=7mm]
  \node[isnode,on chain] (l1) [label={above:$G_6$}]{};
  \node[isnode,on chain] (l2) {};
  \node[isnode,on chain] (l3) {};
\end{scope}

\draw (f1) edge[blue,bend left] node [left] {}(h1);
\draw (h1) edge[gray,bend left] node [left] {}(g2);
\draw (g2) edge[red,bend left] node [left] {}(i3);
\draw (i3) edge[green,bend left] node [left] {}(f2);
\draw (f2) edge[blue,bend left] node [left] {}(h2);
\draw (h2) edge[gray,bend left] node [left] {}(g3);
\draw (g3) edge[red,bend left] node [left] {}(i1);
\draw (i1) edge[green,bend left] node [left] {}(f3);
\draw (f3) edge[blue,bend left] node [left] {}(h3);
\draw (h3) edge[gray,bend left] node [left] {}(g1);
\draw (g1) edge[red,bend left] node [left] {}(i2);
\draw (i2) edge[green,bend left] node [left] {}(f1);
\draw (j1) edge[blue,bend left] node [left] {}(k2);
\draw (k2) edge[red,bend left] node [left] {}(l3);
\draw (l3) edge[gray,bend left] node [left] {}(j1);
\draw (j2) edge[blue,bend left] node [left] {}(k3);
\draw (k3) edge[red,bend left] node [left] {}(l1);
\draw (l1) edge[gray,bend left] node [left] {}(j2);
\draw (j3) edge[blue,bend left] node [left] {}(k1);
\draw (k1) edge[red,bend left] node [left] {}(l2);
\draw (l2) edge[gray,bend left] node [left] {}(j3);
\end{tikzpicture}

\caption{Edge induced subgraph with cross differences $(0,1,1,2)$ on the  4-cycle $(0,2,1,3)$, and cross differences $(1,1,1)$ on the 3-cycle $(4,5,6)$}.
\label{0}
\end{figure}

\begin{lemma}
\label{Lemma2}
Let $G=K_{(g:x)}$, and let $H$ be an $x$-cycle. Given the list of differences $(d_1,d_2, \ldots ,d_x)$, let $d \equiv \sum_{k=1}^x d_k \pmod{g}$.  If $\gcd(d, g)=h$, then $G[\Delta]$ consists of $h$ components; and each component is
\begin{enumerate}
\item{a ${(\frac{g}{h})}x$-cycle if $d \neq \frac{g}{2}$ and}
\item{a $2x$-cycle if $d=\frac{g}{2}$.}
\end{enumerate}
\end{lemma}

\begin{proof}
Without loss of generality, we assume the cycle $c=(0,1, \ldots, x-1)$ is given. Suppose $d \not = \frac{g}{2}$. Let $D=\{d\}$, and consider the graph $G(D,g)$ on $G_0$. By Corollary ~\ref{Stern-Lenz2}, $G(D,g)$ consists of $h$ components, and each component is a $(\frac{g}{h})$-cycle. 

For every edge $\{(0,j_1),(0,j_2)\}$ on this $(\frac{g}{h})$-cycle, there is a path of length $x$ on $G[\Delta]$ that connects $(0,j_1)$ to $(0,j_2)$. Thus, any cycle of length $(\frac{g}{h})$ on $G_0$ corresponds to a cycle of length $(\frac{g}{h})x$ in $G[\Delta]$.

Now suppose $d=\frac{g}{2}$. Then by Corollary~\ref{Stern-Lenz}, $G(D,g)$ consists of $\frac{g}{2}$ components, and each component is a $K_2$. For every edge $\{(0,j_1),(0,j_2)\}$, there is a path of length $x$ from $(0,j_1)$ to $(0,j_2)$ in $G[\Delta]$, and there is a path of length $x$ from $(0,j_2)$ to $(0,j_1)$ in $G[\Delta]$. Therefore, each edge in $G(D,g)$ corresponds to a cycle of length $2x$ in $G[\Delta]$ .
\end{proof}

\

\section{Modified Row Sum Matrices}

When decomposing $G=K_{(g:x)}$ into cycle factors, it is helpful to be able to keep track of which cross differences have been covered, and what cycle sizes these differences produce in $G[\Delta]$. The concept of a row-sum matrix was introduced in \cite{BUR}. For our purposes, we use a similar structure which we will refer to as a modified row-sum matrix.

Let $\Gamma$ be a group, and let $S \subset \Gamma$. For $i= 0,1, \ldots,x-1$, let $S_i \subseteq S$, where $|S_i|=t$. Also, let $\Sigma$ be a $t$-list of elements of $\Gamma$. A modified row-sum matrix $MRSM_\Gamma(S,t,x;\Sigma)$ is a $t$ by $x$ matrix, whose $x \geq 2$ columns are permutations of $S_i$ and such that the list of (left-to-right) row-sums is $\Sigma$.

Given a $MRSM_{\mathbb Z_g}(S,t,x;\Sigma)$, each row of this matrix produces a subgraph, $G[\Delta]$, of $G=K_{(g:x)}$, according to part 1 of Lemma \ref{Lemma2}.

\begin{corollary}
    Suppose $(d_{i_1},d_{i_2},\ldots d_{i_x})$ is a row of a modified row-sum matrix \\ $MRSM_{\mathbb Z_g}(S,t,x;\Sigma)$. If $d\equiv \sum_{k=1}^{x} d_{i_{k}} \pmod{g}$ with $d \neq \frac{g}{2}$, then $G[\Delta]$ consists of $h=\gcd(d,g)$ components, and each component is a $\big (\frac{g}{h} \big)x$-cycle.
\end{corollary}

\ 

Denote ${d_k}^+$ as the union of an edge $((i_k,j),(i_{k+1},j+d_k))$ with cross difference $d_k$, and the inside edge $\{(i_{k+1},j+d_k),(i_{k+1},j+d_k+1)\}$ with difference 1. Then ${d_k}^+$ consists of a 2-path whose endpoints have a cross difference of $d_{k}+1$. Similarly, denote ${d_k}^-$ as the union of an edge $((i_k,j),(i_{k+1},j+d_k))$ with cross difference $d_k$, and the inside edge $\{(i_{k+1},j+d_k),(i_{k+1},j+d_k-1)\}$ with difference $-1$. Then ${d_k}^-$ consists of a 2-path whose endpoints have a cross difference of $d_{k}-1$. We provide a representation of ${d_k}^+$ and ${d_k}^-$ in Figure~\ref{1}.

\begin{figure}[h!]
\centering
\definecolor{mygray}{RGB}{70,80,70}

\begin{tikzpicture}[thick,
  every node/.style={circle, inner sep=0.1cm},
  fsnode/.style={fill=mygray},
  gsnode/.style={fill=mygray},
  hsnode/.style={fill=mygray},
  isnode/.style={fill=mygray}
]

\begin{scope}[start chain=going below,node distance=7mm]
  \node[fsnode,on chain] (f1) [label={left:${(i_k,j)}$}]{};
  \node[fsnode,on chain] (f2) {};
  \node[fsnode,on chain] (f3) {};
  \node[fsnode,on chain] (f4) {};
  \node[fsnode,on chain] (f5) {};
  \node[fsnode,on chain] (f6) {};
\end{scope}

\begin{scope}[xshift=2cm,yshift=0cm,start chain=going below,node distance=7mm]
  \node[gsnode,on chain] (g1) [label={left:${d_{k}}^+$}] {};
  \node[gsnode,on chain] (g2) {};
  \node[gsnode,on chain] (g3) [label={right:$(i_{k+1},j+d_k)$}] {};
  \node[gsnode,on chain] (g4) [label={right:$(i_{k+1},j+d_k+1)$}] {};
  \node[gsnode,on chain] (g5) {};
  \node[gsnode,on chain] (g6) {};
\end{scope}

\begin{scope}[xshift=6cm,yshift=0cm,start chain=going below,node distance=7mm]
  \node[hsnode,on chain] (h1) [label={left:${(i_k,j)}$}]{};
  \node[hsnode,on chain] (h2) {};
  \node[hsnode,on chain] (h3) {};
  \node[hsnode,on chain] (h4) {};
  \node[hsnode,on chain] (h5) {};
  \node[hsnode,on chain] (h6) {};
\end{scope}

\begin{scope}[xshift=8cm,yshift=0cm,start chain=going below,node distance=7mm]
  \node[isnode,on chain] (i1) [label={left:${d_{k}}^-$}] {};
  \node[isnode,on chain] (i2) [label={right:$(i_{k+1},j+d_k-1)$}]{};
  \node[isnode,on chain] (i3) [label={right:$(i_{k+1},j+d_k)$}] {};
  \node[isnode,on chain] (i4) {};
  \node[isnode,on chain] (i5) {};
  \node[isnode,on chain] (i6) {};
\end{scope}

\draw (f1) edge[black,line width= 1pt] node {}(g3);
\draw (g3) -- (g4);
\draw (f1) edge[red,line width= 1.5pt] [dashed] node {} (g4);
\draw (h1) edge[black,line width= 1pt] node {} (i3);
\draw (i3) -- (i2);
\draw (h1) edge[red,line width= 1.5pt] [dashed] node {} (i2);
\end{tikzpicture}

\caption{Representation of ${d_k}^+$ and ${d_k}^-$.}
\label{1}
\end{figure}
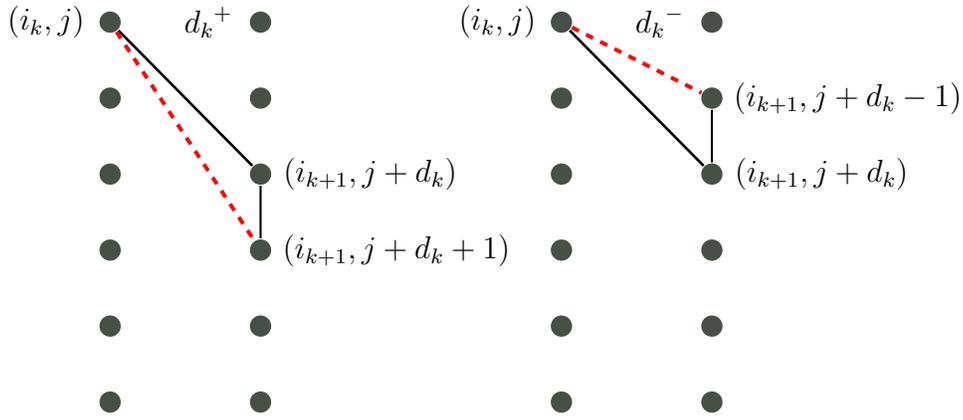

\

It is well known that a 1-factorization of $K_6$ exists. We give a specific 1-factorization of $K_6$ into five 1-factors; $f_1, \ldots, f_5$, in Figure \ref{1-factors}; this particular 1-factorization will be used to prove the following.

\begin{figure}[h!]
\centering
\definecolor{mygray}{RGB}{70,80,70}

\begin{tikzpicture}[thick,
  every node/.style={circle, inner sep=0.1cm},
  fsnode/.style={fill=mygray},
  gsnode/.style={fill=mygray},
  hsnode/.style={fill=mygray},
  isnode/.style={fill=mygray},
  jsnode/.style={fill=mygray}
]

\begin{scope}[start chain=going below,node distance=7mm]
\foreach \i in {1,2,...,6}
  \node[fsnode,on chain] (f\i) {};
\end{scope}

\begin{scope}[xshift=2cm,yshift=0cm,start chain=going below,node distance=7mm]
\foreach \i in {1,2,...,6}
  \node[gsnode,on chain] (g\i) {};
\end{scope}

\begin{scope}[xshift=4cm,yshift=0cm,start chain=going below,node distance=7mm]
\foreach \i in {1,2,...,6}
  \node[hsnode,on chain] (h\i) {};
\end{scope}

\begin{scope}[xshift=6cm,yshift=0cm,start chain=going below,node distance=7mm]
\foreach \i in {1,2,...,6}
  \node[isnode,on chain] (i\i) {};
\end{scope}

\begin{scope}[xshift=8cm,yshift=0cm,start chain=going below,node distance=7mm]
\foreach \i in {1,2,...,6}
  \node[isnode,on chain] (j\i) {};
\end{scope}

\node [fit=(f1) (f6),label=above:$f_1$] {};
\node [fit=(g1) (g6),label=above:$f_2$] {};
\node [fit=(h1) (h6),label=above:$f_3$] {};
\node [fit=(i1) (i6),label=above:$f_4$] {};
\node [fit=(j1) (j6),label=above:$f_5$] {};

\draw (f1) -- (f2);
\draw (f3) -- (f4);
\draw (f5) -- (f6);
\draw (g1) edge[bend right] node [left] {}(g6);
\draw (g2) -- (g3);
\draw (g4) -- (g5);
\draw (h1) edge[bend right] node [left] {}(h4);
\draw (h2) edge[bend right] node [left] {}(h6);
\draw (h3) edge[bend right] node [left] {}(h5);
\draw (i1) edge[bend right] node [left] {}(i5);
\draw (i2) edge[bend right] node [left] {}(i4);
\draw (i3) edge[bend right] node [left] {}(i6);
\draw (j1) edge[bend right] node [left] {}(j3);
\draw (j2) edge[bend right] node [left] {}(j5);
\draw (j4) edge[bend right] node [left] {}(j6);
\end{tikzpicture}
\caption{1-factorization of $K_6$ into 5 1-factors.}
\label{1-factors}
\end{figure}
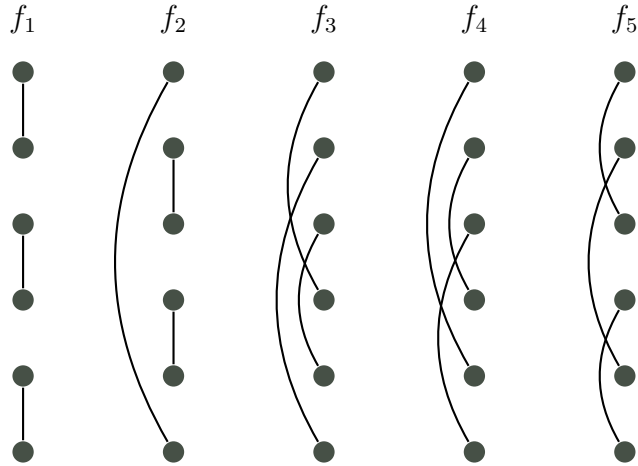

\begin{lemma}
\label{mix}
Let $T$ be a $MRSM_{\mathbb Z_6}(S,t,x;\Sigma)$ with $x$ even and with the property that the entries in some row $\alpha$ sum to $d \pmod{6}$, with $d$ even. Then $G[\Delta] {\bigcup_{k=1}^{x}} G_{i_k}[f_1 \bigcup f_2]$ has a decomposition into two $(2x)$-cycle factors. 
\end{lemma}
\begin{proof}

{\bf Case 1.} Suppose the entries in row $\alpha$, of $T$ sum to $d\equiv 0 \pmod{6}$. Without loss of generality, we may assume that row $\alpha=(d_{i_1},d_{i_2}, \ldots, d_{i_x})=(d_0,d_1, \ldots, d_{x-1})$. For $k \in 0,1, \ldots, x-1$, replace each entry $d_k$ with $d_k^+$ if $k$ is even and $d_k^-$ if $k$ is odd. Then any edge from $G[\Delta]$ with cross difference $d_k$ corresponds to a 2-path whose endpoints have a cross difference of $d_k +1$ (for $d_k^+$) or $d_k -1$ (for $d_k^-$). Let $c'_0$ be the cycle containing point $(0,0)$ obtained by joining the endpoints of these 2-paths. The length of $c'_0$ can be determined by computing the sum of the revised entries in row $\alpha$ where $d_k^+ = d_k +1$ and $d_k^- = d_k -1$. This sum is 
\begin{align*} 
 d_1 +1+d_2 -1+d_3 +1 + \ldots + d_x -1 & = d_1+d_2+ \ldots+d_x +1-1 \ldots +1-1 \\
 & = d_1+d_2+ \ldots+d_x \\
 & = d \\
 & \equiv 0 \pmod{6}. 
\end{align*}

Therefore, by Lemma ~\ref{Lemma2}, $c'_0$ has length $x$.

\

We can now form a $2x$-cycle, $c_0$, by replacing each edge from $c'_0$ with its corresponding 2-path with the same endpoints. Thus $c_0$ contains cross-edges from $G[\Delta]$ and inside edges from $G_{i_k}[f_1 \cup f_2]$ for $k=1, \ldots, x$.

We can obtain two more $2x$-cycles, $c_2$ and $c_4$, by repeating this procedure with the cycles $c'_2$ (containing point $(0,2)$) and $c'_4$ (containing point $(0,4)$). This produces a $2x$-cycle factor.

We obtain a second factor by repeating this process with the cycles $c'_1, c'_3, c'_5$ which contain the points $(0,1)$, $(0,3)$, and $(0,5)$ respectively.  These two $2x$-cycle factors decompose $G[\Delta] \cup G_{i_k}[f_1 \cup f_2]$ for $k=1,2, \ldots, x$.

{\bf Case 2.} Suppose the entries in some row $\alpha$, of $T$ sum to $d\equiv 2 \pmod{6}$. Replace the first $x-2$ entries of row $\alpha$ with alternating $d_k^+$ and $d_k^-$, and the last two entries with $d_k^-$. 

Letting $d_k^+=d_k +1$ and $d_k^-=d_k -1$, we obtain the sum of row $\alpha$ as 

\begin{multline*}
     d_0 +1+d_1 -1+ \ldots + d_{x-4} +1+ d_{x-3} -1 +d_{x-2} -1+ d_{x-1}-1 \\
     = d_0+d_1+ \ldots+d_{x-3} + d_{x-2} -1+ d_{x-1} -1 = d-2 \equiv 0 \pmod{6}.
\end{multline*}

Now follow the technique given in case 1 to form two $2x$-cycle factors.

\

{\bf Case 3.} When $d\equiv 4 \pmod{6}$, replace the first $x-4$ entries of row $\alpha$  with alternating ${d_k}^+$ and ${d_k}^-$, and replace the last four entries with ${d_k}^-$ to obtain a sum that is congruent to $d\equiv 0 \pmod{6}$. Again, apply the same algorithm described in case 1 to obtain two $2x$-cycle factors. 

\end{proof}

 An illustration of this lemma, when $x=4$, is given in Figure \ref{4}.

\

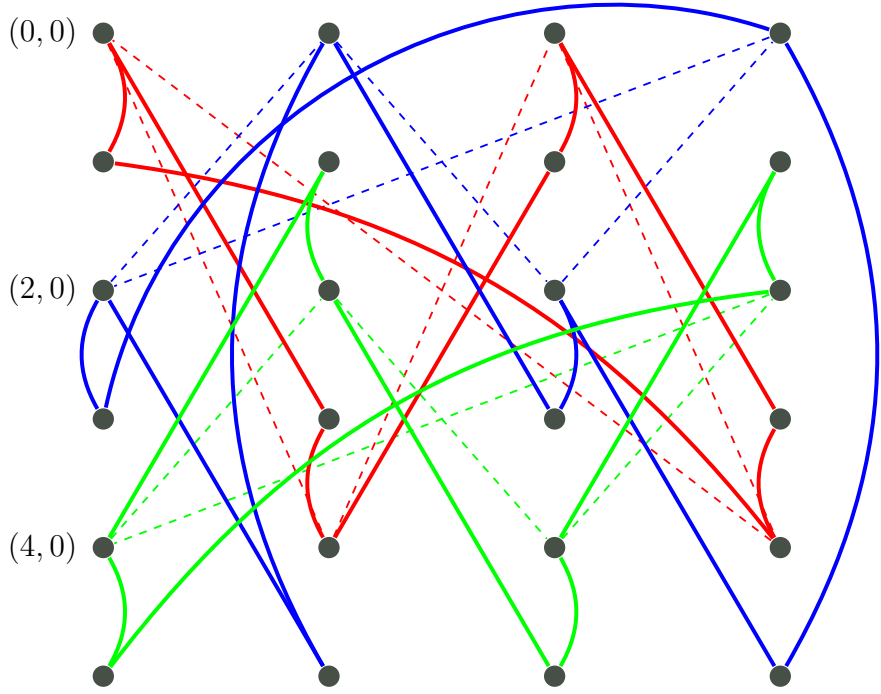
\begin{figure}[h!]
\centering
\definecolor{mygray}{RGB}{70,80,70}

\begin{tikzpicture}[thick,
  every node/.style={circle, inner sep=0.1cm},
  fsnode/.style={fill=mygray},
  gsnode/.style={fill=mygray},
  hsnode/.style={fill=mygray},
  isnode/.style={fill=mygray}
]

\begin{scope}[start chain=going below,node distance=14mm]
  \node[fsnode,on chain] (f1) [label={left:${(0,0)}$}] {};
  \node[fsnode,on chain] (f2) {};
  \node[fsnode,on chain] (f3) [label={left:${(2,0)}$}] {};
  \node[fsnode,on chain] (f4) {};
  \node[fsnode,on chain] (f5) [label={left:${(4,0)}$}] {};
  \node[fsnode,on chain] (f6) {};
\end{scope}

\begin{scope}[xshift=3cm,yshift=0cm,start chain=going below,node distance=14mm]
  \node[gsnode,on chain] (g1) {};
  \node[gsnode,on chain] (g2) {};
  \node[gsnode,on chain] (g3) {};
  \node[gsnode,on chain] (g4) {};
  \node[gsnode,on chain] (g5) {};
  \node[gsnode,on chain] (g6) {};
\end{scope}

\begin{scope}[xshift=6cm,yshift=0cm,start chain=going below,node distance=14mm]
  \node[hsnode,on chain] (h1) {};
  \node[hsnode,on chain] (h2) {};
  \node[hsnode,on chain] (h3) {};
  \node[hsnode,on chain] (h4) {};
  \node[hsnode,on chain] (h5) {};
  \node[hsnode,on chain] (h6) {};
\end{scope}

\begin{scope}[xshift=9cm,yshift=0cm,start chain=going below,node distance=14mm]
  \node[isnode,on chain] (i1) {};
  \node[isnode,on chain] (i2) {};
  \node[isnode,on chain] (i3) {};
  \node[isnode,on chain] (i4) {};
  \node[isnode,on chain] (i5) {};
  \node[isnode,on chain] (i6) {};
\end{scope}

\draw (f1) edge[red,line width= 0.7pt] [dashed] node {} (g5);
\draw (g5) edge[red,line width= 0.7pt] [dashed] node {} (h1);
\draw (h1) edge[red,line width= 0.7pt] [dashed] node {} (i5);
\draw (i5) edge[red,line width= 0.7pt] [dashed] node {} (f1);
\draw (f3) edge[blue,line width= 0.7pt] [dashed] node {} (g1);
\draw (g1) edge[blue,line width= 0.7pt] [dashed] node {} (h3);
\draw (h3) edge[blue,line width= 0.7pt] [dashed] node {} (i1);
\draw (i1) edge[blue,line width= 0.7pt] [dashed] node {} (f3);
\draw (f5) edge[green,line width= 0.7pt] [dashed] node {} (g3);
\draw (g3) edge[green,line width= 0.7pt] [dashed] node {} (h5);
\draw (h5) edge[green,line width= 0.7pt] [dashed] node {} (i3);
\draw (i3) edge[green,line width= 0.7pt] [dashed]node {} (f5);
\draw (f1) edge[red,line width= 1.5pt] node {} (g4);
\draw (g4) edge[red,line width= 1.5pt] [bend right] node [left] {} (g5);
\draw (g5) edge[red,line width= 1.5pt] node {} (h2);
\draw (h2) edge[red,line width= 1.5pt] [bend right] node [left] {} (h1);
\draw (h1) edge[red,line width= 1.5pt] node {} (i4);
\draw (i4) edge[red,line width= 1.5pt] [bend right] node [left] {} (i5);
\draw (i5) edge[red,line width= 1.5pt] [bend right=22] node {} (f2);
\draw (f2) edge[red,line width= 1.5pt] [bend right] node [left] {} (f1);
\draw (f3) edge[blue,line width= 1.5pt] node {} (g6);
\draw (g6) edge[blue,line width= 1.5pt] [bend left] node [left] {} (g1);
\draw (g1) edge[blue,line width= 1.5pt] node {} (h4);
\draw (h4) edge[blue,line width= 1.5pt] [bend right] node [left] {} (h3);
\draw (h3) edge[blue,line width= 1.5pt] node {} (i6);
\draw (i6) edge[blue,line width= 1.5pt] [bend right] node [left] {} (i1);
\draw (i1) edge[blue,line width= 1.5pt] [bend right=48] node {} (f4);
\draw (f4) edge[blue,line width= 1.5pt] [bend left] node [left] {} (f3);
\draw (f5) edge[green,line width= 1.5pt] node {} (g2);
\draw (g2) edge[green,line width= 1.5pt] [bend right] node [left] {} (g3);
\draw (g3) edge[green,line width= 1.5pt] node {} (h6);
\draw (h6) edge[green,line width= 1.5pt] [bend right] node [left] {} (h5);
\draw (h5) edge[green,line width= 1.5pt] node {} (i2);
\draw (i2) edge[green,line width= 1.5pt] [bend right] node [left] {} (i3);
\draw (i3) edge[green,line width= 1.5pt] [bend right=23] node {} (f6);
\draw (f6) edge[green,line width= 1.5pt] [bend right] node [left] {} (f5);
\end{tikzpicture}

\caption{One 8-cycle factor in $K_{(6:4)}$ from a $MRSM_{\mathbb Z_6}(S,t,4;\Sigma)$ with row $\alpha = [3^+,3^-,3^+,3^-]$. The dashed edges represent the 4-cycles ${c'_0}, c'_2, c'_4$.}
\label{4}
\end{figure}
\

\begin{lemma}
\label{row0s} 
Let $T$ be a $MRSM_{\mathbb Z_6}(S,t,x;\Sigma)$ with $x$ even and with the property that the entries in some row are all $0$. Then $G[\Delta] {\bigcup_{k=1}^{x}} G_{i_{k}}[f_5]$ can be decomposed into a $(2x)$-cycle and a 1-factor.
\end{lemma}
\begin{proof}

\
If the entries in some row of $T$ are all $0$, then by Lemma~\ref{Lemma2}, $G[\Delta]$ consists of six components $c_1, c_2, \dots, c_6$, and each component is an $x$-cycle that uses all cross edges of differences zero. However, we want to form $2x$-cycles instead. For that purpose, we will also include the inside edges from $f_5$ given in Figure~\ref{1-factors} to form $2x$-cycles. 

As shown in Figure~\ref{in+0}, we use alternating edges from the cycles $c_1, c_2, \dots, c_6$ and the inside edges from $f_5$. 
The leftover edges from the cycles $c_1, \dots, c_6$ produce a 1-factor as shown in Figure ~\ref{1factor}.

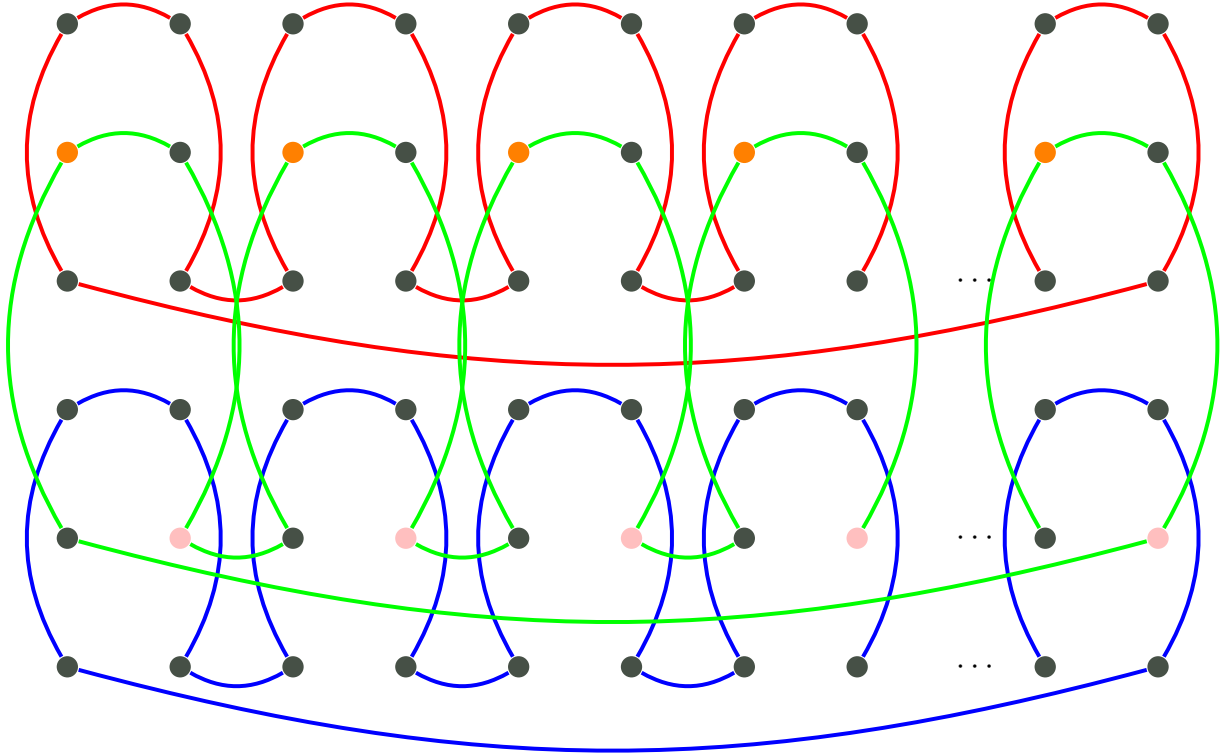
\begin{figure}[h!]
\centering
\definecolor{mygray}{RGB}{70,80,70}

\begin{tikzpicture}[thick,
  every node/.style={circle, inner sep=0.1cm},
  fsnode/.style={fill=mygray},
  gsnode/.style={fill=mygray},
  hsnode/.style={fill=mygray},
  isnode/.style={fill=mygray},
  jsnode/.style={fill=mygray},
  ksnode/.style={fill=mygray},
  lsnode/.style={fill=mygray},
  msnode/.style={fill=mygray},
  nsnode/.style={fill=mygray},
  osnode/.style={fill=mygray}
]

\begin{scope}[start chain=going below,node distance=14mm]
  \node[fsnode,on chain] (f1) {};
  \node[fsnode,on chain] (f2) [fill=orange] {};
  \node[fsnode,on chain] (f3) {};
  \node[fsnode,on chain] (f4) {};
  \node[fsnode,on chain] (f5) {};
  \node[fsnode,on chain] (f6) {};
\end{scope}

\begin{scope}[xshift=1.5cm,yshift=0cm,start chain=going below,node distance=14mm]
  \node[gsnode,on chain] (g1) {};
  \node[gsnode,on chain] (g2) {};
  \node[gsnode,on chain] (g3) {};
  \node[gsnode,on chain] (g4) {};
  \node[gsnode,on chain] (g5) [fill=pink] {};
  \node[gsnode,on chain] (g6) {};
\end{scope}

\begin{scope}[xshift=3cm,yshift=0cm,start chain=going below,node distance=14mm]
  \node[hsnode,on chain] (h1) {};
  \node[hsnode,on chain] (h2) [fill=orange] {};
  \node[hsnode,on chain] (h3) {};
  \node[hsnode,on chain] (h4) {};
  \node[hsnode,on chain] (h5) {};
  \node[hsnode,on chain] (h6) {};
\end{scope}

\begin{scope}[xshift=4.5cm,yshift=0cm,start chain=going below,node distance=14mm]
  \node[isnode,on chain] (i1) {};
  \node[isnode,on chain] (i2) {};
  \node[isnode,on chain] (i3) {};
  \node[isnode,on chain] (i4) {};
  \node[isnode,on chain] (i5) [fill=pink] {};
  \node[isnode,on chain] (i6) {};
\end{scope}

\begin{scope}[xshift=6cm,yshift=0cm,start chain=going below,node distance=14mm]
  \node[jsnode,on chain] (j1) {};
  \node[jsnode,on chain] (j2) [fill=orange] {};
  \node[jsnode,on chain] (j3) {};
  \node[jsnode,on chain] (j4) {};
  \node[jsnode,on chain] (j5) {};
  \node[jsnode,on chain] (j6) {};
\end{scope}

\begin{scope}[xshift=7.5cm,yshift=0cm,start chain=going below,node distance=14mm]
  \node[ksnode,on chain] (k1) {};
  \node[ksnode,on chain] (k2) {};
  \node[ksnode,on chain] (k3) {};
  \node[ksnode,on chain] (k4) {};
  \node[ksnode,on chain] (k5) [fill=pink] {};
  \node[ksnode,on chain] (k6) {};
\end{scope}

\begin{scope}[xshift=9cm,yshift=0cm,start chain=going below,node distance=14mm]
  \node[lsnode,on chain] (l1) {};
  \node[lsnode,on chain] (l2) [fill=orange] {};
  \node[lsnode,on chain] (l3) {};
  \node[lsnode,on chain] (l4) {};
  \node[lsnode,on chain] (l5) {};
  \node[lsnode,on chain] (l6) {};
\end{scope}

\begin{scope}[xshift=10.5cm,yshift=0cm,start chain=going below,node distance=14mm]
  \node[msnode,on chain] (m1) {};
  \node[msnode,on chain] (m2) {};
  \node[msnode,on chain] (m3) {};
  \node[msnode,on chain] (m4) {};
  \node[msnode,on chain] (m5) [fill=pink] {};
  \node[msnode,on chain] (m6) {};
\end{scope}

\begin{scope}[xshift=13cm,yshift=0cm,start chain=going below,node distance=14mm]
  \node[nsnode,on chain] (n1) {};
  \node[nsnode,on chain] (n2) [fill=orange]{};
  \node[nsnode,on chain] (n3) {};
  \node[nsnode,on chain] (n4) {};
  \node[nsnode,on chain] (n5) {};
  \node[nsnode,on chain] (n6) {};
\end{scope}

\begin{scope}[xshift=14.5cm,yshift=0cm,start chain=going below,node distance=14mm]
  \node[osnode,on chain] (o1) {};
  \node[osnode,on chain] (o2) {};
  \node[osnode,on chain] (o3) {};
  \node[osnode,on chain] (o4) {};
  \node[osnode,on chain] (o5) [fill=pink] {};
  \node[osnode,on chain] (o6) {};
\end{scope}


\node[right=of m3] {$\cdots$};
\node[right=of m5] {$\cdots$};
\node[right=of m6] {$\cdots$};

\draw (f1) edge[red,line width= 1.5pt] [bend left] node {} (g1);
\draw (g1) edge[red,line width= 1.5pt] [bend left] node {} (g3);
\draw (g3) edge[red,line width= 1.5pt] [bend right] node {} (h3);
\draw (h3) edge[red,line width= 1.5pt] [bend left] node {} (h1);
\draw (h1) edge[red,line width= 1.5pt] [bend left] node {} (i1);
\draw (i1) edge[red,line width= 1.5pt] [bend left] node {} (i3);
\draw (i3) edge[red,line width= 1.5pt] [bend right] node {} (j3);
\draw (j3) edge[red,line width= 1.5pt] [bend left] node {} (j1);
\draw (j1) edge[red,line width= 1.5pt] [bend left] node {} (k1);
\draw (k1) edge[red,line width= 1.5pt] [bend left] node {} (k3);
\draw (k3) edge[red,line width= 1.5pt] [bend right] node {} (l3);
\draw (l3) edge[red,line width= 1.5pt] [bend left] node {} (l1);
\draw (l1) edge[red,line width= 1.5pt] [bend left] node {} (m1);
\draw (m1) edge[red,line width= 1.5pt] [bend left] node {} (m3);
\draw (n3) edge[red,line width= 1.5pt] [bend left] node {} (n1);
\draw (n1) edge[red,line width= 1.5pt] [bend left] node {} (o1);
\draw (o1) edge[red,line width= 1.5pt] [bend left] node {} (o3);
\draw (o3) edge[red,line width= 1.5pt] [bend left=15] node {} (f3);
\draw (f3) edge[red,line width= 1.5pt] [bend left] node {} (f1);
\draw (f4) edge[blue,line width= 1.5pt] [bend left] node {} (g4);
\draw (g4) edge[blue,line width= 1.5pt] [bend left] node {} (g6);
\draw (g6) edge[blue,line width= 1.5pt] [bend right] node {} (h6);
\draw (h6) edge[blue,line width= 1.5pt] [bend left] node {} (h4);
\draw (h4) edge[blue,line width= 1.5pt] [bend left] node {} (i4);
\draw (i4) edge[blue,line width= 1.5pt] [bend left] node {} (i6);
\draw (i6) edge[blue,line width= 1.5pt] [bend right] node {} (j6);
\draw (j6) edge[blue,line width= 1.5pt] [bend left] node {} (j4);
\draw (j4) edge[blue,line width= 1.5pt] [bend left] node {} (k4);
\draw (k4) edge[blue,line width= 1.5pt] [bend left] node {} (k6);
\draw (k6) edge[blue,line width= 1.5pt] [bend right] node {} (l6);
\draw (l6) edge[blue,line width= 1.5pt] [bend left] node {} (l4);
\draw (l4) edge[blue,line width= 1.5pt] [bend left] node {} (m4);
\draw (m4) edge[blue,line width= 1.5pt] [bend left] node {} (m6);
\draw (n6) edge[blue,line width= 1.5pt] [bend left] node {} (n4);
\draw (n4) edge[blue,line width= 1.5pt] [bend left] node {} (o4);
\draw (o4) edge[blue,line width= 1.5pt] [bend left] node {} (o6);
\draw (o6) edge[blue,line width= 1.5pt] [bend left=15] node {} (f6);
\draw (f6) edge[blue,line width= 1.5pt] [bend left] node {} (f4);
\draw (f2) edge[green,line width= 1.5pt] [bend left] node {} (g2);
\draw (g2) edge[green,line width= 1.5pt] [bend left] node {} (g5);
\draw (g5) edge[green,line width= 1.5pt] [bend right] node {} (h5);
\draw (h5) edge[green,line width= 1.5pt] [bend left] node {} (h2);
\draw (h2) edge[green,line width= 1.5pt] [bend left] node {} (i2);
\draw (i2) edge[green,line width= 1.5pt] [bend left] node {} (i5);
\draw (i5) edge[green,line width= 1.5pt] [bend right] node {} (j5);
\draw (j5) edge[green,line width= 1.5pt] [bend left] node {} (j2);
\draw (j2) edge[green,line width= 1.5pt] [bend left] node {} (k2);
\draw (k2) edge[green,line width= 1.5pt] [bend left] node {} (k5);
\draw (k5) edge[green,line width= 1.5pt] [bend right] node {} (l5);
\draw (l5) edge[green,line width= 1.5pt] [bend left] node {} (l2);
\draw (l2) edge[green,line width= 1.5pt] [bend left] node {} (m2);
\draw (m2) edge[green,line width= 1.5pt] [bend left] node {} (m5);
\draw (n5) edge[green,line width= 1.5pt] [bend left] node {} (n2);
\draw (n2) edge[green,line width= 1.5pt] [bend left] node {} (o2);
\draw (o2) edge[green,line width= 1.5pt] [bend left] node {} (o5);
\draw (o5) edge[green,line width= 1.5pt] [bend left=15] node {} (f5);
\draw (f5) edge[green,line width= 1.5pt] [bend left] node {} (f2);
\end{tikzpicture}
\caption{A $2x$-cycle factor formed from cross differences zero and inside edges from $f_5$.}
\label{in+0}
\end{figure}

\begin{figure}[h!]
\centering
\definecolor{mygray}{RGB}{70,80,70}

\begin{tikzpicture}[thick,
  every node/.style={circle, inner sep=0.1cm},
  fsnode/.style={fill=mygray},
  gsnode/.style={fill=mygray},
  hsnode/.style={fill=mygray},
  isnode/.style={fill=mygray},
  jsnode/.style={fill=mygray},
  ksnode/.style={fill=mygray},
  lsnode/.style={fill=mygray},
  msnode/.style={fill=mygray},
  nsnode/.style={fill=mygray},
  osnode/.style={fill=mygray}
]

\begin{scope}[start chain=going below,node distance=14mm]
  \node[fsnode,on chain] (f1) {};
  \node[fsnode,on chain] (f2) {};
  \node[fsnode,on chain] (f3) {};
  \node[fsnode,on chain] (f4) {};
  \node[fsnode,on chain] (f5) {};
  \node[fsnode,on chain] (f6) {};
\end{scope}

\begin{scope}[xshift=1.5cm,yshift=0cm,start chain=going below,node distance=14mm]
  \node[gsnode,on chain] (g1) {};
  \node[gsnode,on chain] (g2) {};
  \node[gsnode,on chain] (g3) {};
  \node[gsnode,on chain] (g4) {};
  \node[gsnode,on chain] (g5) {};
  \node[gsnode,on chain] (g6) {};
\end{scope}

\begin{scope}[xshift=3cm,yshift=0cm,start chain=going below,node distance=14mm]
  \node[hsnode,on chain] (h1) {};
  \node[hsnode,on chain] (h2) {};
  \node[hsnode,on chain] (h3) {};
  \node[hsnode,on chain] (h4) {};
  \node[hsnode,on chain] (h5) {};
  \node[hsnode,on chain] (h6) {};
\end{scope}

\begin{scope}[xshift=4.5cm,yshift=0cm,start chain=going below,node distance=14mm]
  \node[isnode,on chain] (i1) {};
  \node[isnode,on chain] (i2) {};
  \node[isnode,on chain] (i3) {};
  \node[isnode,on chain] (i4) {};
  \node[isnode,on chain] (i5) {};
  \node[isnode,on chain] (i6) {};
\end{scope}

\begin{scope}[xshift=6cm,yshift=0cm,start chain=going below,node distance=14mm]
  \node[jsnode,on chain] (j1) {};
  \node[jsnode,on chain] (j2) {};
  \node[jsnode,on chain] (j3) {};
  \node[jsnode,on chain] (j4) {};
  \node[jsnode,on chain] (j5) {};
  \node[jsnode,on chain] (j6) {};
\end{scope}

\begin{scope}[xshift=7.5cm,yshift=0cm,start chain=going below,node distance=14mm]
  \node[ksnode,on chain] (k1) {};
  \node[ksnode,on chain] (k2) {};
  \node[ksnode,on chain] (k3) {};
  \node[ksnode,on chain] (k4) {};
  \node[ksnode,on chain] (k5) {};
  \node[ksnode,on chain] (k6) {};
\end{scope}

\begin{scope}[xshift=9cm,yshift=0cm,start chain=going below,node distance=14mm]
  \node[lsnode,on chain] (l1) {};
  \node[lsnode,on chain] (l2) {};
  \node[lsnode,on chain] (l3) {};
  \node[lsnode,on chain] (l4) {};
  \node[lsnode,on chain] (l5) {};
  \node[lsnode,on chain] (l6) {};
\end{scope}

\begin{scope}[xshift=10.5cm,yshift=0cm,start chain=going below,node distance=14mm]
  \node[msnode,on chain] (m1) {};
  \node[msnode,on chain] (m2) {};
  \node[msnode,on chain] (m3) {};
  \node[msnode,on chain] (m4) {};
  \node[msnode,on chain] (m5) {};
  \node[msnode,on chain] (m6) {};
\end{scope}

\begin{scope}[xshift=12cm,yshift=0cm,start chain=going below,node distance=14mm]
  \node[nsnode,on chain] (n1) {};
  \node[nsnode,on chain] (n2) {};
  \node[nsnode,on chain] (n3) {};
  \node[nsnode,on chain] (n4) {};
  \node[nsnode,on chain] (n5) {};
  \node[nsnode,on chain] (n6) {};
\end{scope}

\begin{scope}[xshift=14.5cm,yshift=0cm,start chain=going below,node distance=14mm]
  \node[osnode,on chain] (o1) {};
  \node[osnode,on chain] (o2) {};
  \node[osnode,on chain] (o3) {};
  \node[osnode,on chain] (o4) {};
  \node[osnode,on chain] (o5) {};
  \node[osnode,on chain] (o6) {};
\end{scope}


\node[right=of n1] {$\cdots$};
\node[right=of n2] {$\cdots$};
\node[right=of n3] {$\cdots$};
\node[right=of n4] {$\cdots$};
\node[right=of n5] {$\cdots$};
\node[right=of n6] {$\cdots$};

\draw (g1) edge[red,line width= 1.5pt] [bend left] node {} (h1);
\draw (i1) edge[red,line width= 1.5pt] [bend left] node {} (j1);
\draw (k1) edge[red,line width= 1.5pt] [bend left] node {} (l1);
\draw (m1) edge[red,line width= 1.5pt] [bend left] node {} (n1);
\draw (o1) edge[red,line width= 1.5pt] [bend left=15] node {} (f1);
\draw (g2) edge[red,line width= 1.5pt] [bend left] node {} (h2);
\draw (i2) edge[red,line width= 1.5pt] [bend left] node {} (j2);
\draw (k2) edge[red,line width= 1.5pt] [bend left] node {} (l2);
\draw (m2) edge[red,line width= 1.5pt] [bend left] node {} (n2);
\draw (o2) edge[red,line width= 1.5pt] [bend left=15] node {} (f2);
\draw (f3) edge[red,line width= 1.5pt] [bend left] node {} (g3);
\draw (h3) edge[red,line width= 1.5pt] [bend left] node {} (i3);
\draw (j3) edge[red,line width= 1.5pt] [bend left] node {} (k3);
\draw (l3) edge[red,line width= 1.5pt] [bend left] node {} (m3);
\draw (n3) edge[red,line width= 1.5pt] [bend left] node {} (o3);
\draw (g4) edge[red,line width= 1.5pt] [bend left] node {} (h4);
\draw (i4) edge[red,line width= 1.5pt] [bend left] node {} (j4);
\draw (k4) edge[red,line width= 1.5pt] [bend left] node {} (l4);
\draw (m4) edge[red,line width= 1.5pt] [bend left] node {} (n4);
\draw (o4) edge[red,line width= 1.5pt] [bend left=15] node {} (f4);
\draw (f5) edge[red,line width= 1.5pt] [bend left] node {} (g5);
\draw (h5) edge[red,line width= 1.5pt] [bend left] node {} (i5);
\draw (j5) edge[red,line width= 1.5pt] [bend left] node {} (k5);
\draw (l5) edge[red,line width= 1.5pt] [bend left] node {} (m5);
\draw (n5) edge[red,line width= 1.5pt] [bend left] node {} (o5);
\draw (f6) edge[red,line width= 1.5pt] [bend left] node {} (g6);
\draw (h6) edge[red,line width= 1.5pt] [bend left] node {} (i6);
\draw (j6) edge[red,line width= 1.5pt] [bend left] node {} (k6);
\draw (l6) edge[red,line width= 1.5pt] [bend left] node {} (m6);
\draw (n6) edge[red,line width= 1.5pt] [bend left] node {} (o6);
\end{tikzpicture}
\caption{A 1-factor formed from cross differences zero.}
\label{1factor}
\end{figure}
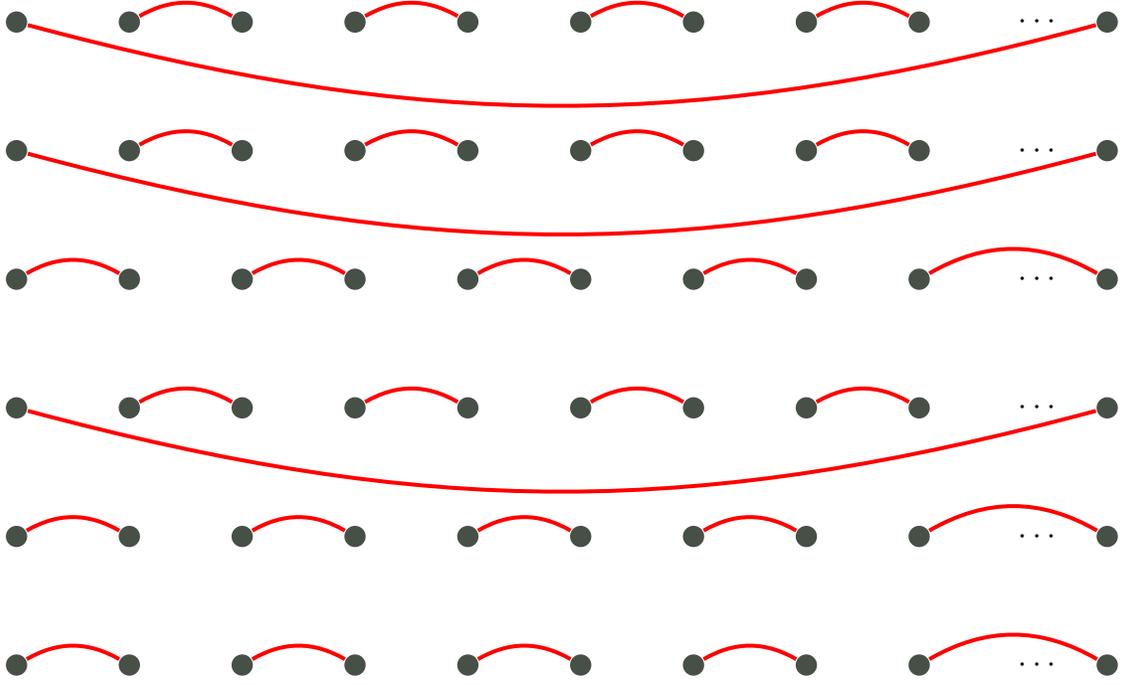

\end{proof}

\section{Building Modified Row Sum Matrices}

In this section, we will build $MRSM$s in a specific way. It will be important to index the columns of our matrices by the edges of a Hamilton cycle. In fact, if $c=(i_1, i_2, \ldots, i_n)$ is a $H$-cycle, then we will index the columns of our $MRSM$ by the edges $i_1i_2, i_2i_3, \ldots, i_ni_1$.

\

Let $C$ be a base cycle in a group $(G, +)$. For any $g \in G$, define $C + g = \{x + g : x \in C\}$. Any cycle $C + g$ is called a translate of $C$. Then, define the development of $C$ to be the collection of all $v$ translates of $C$. The following lemma is a variant of results given in \cite{BER} and \cite{CYLDEC}. 

\begin{lemma}
\label{2Kn}
$2 K_x$ can be decomposed into Hamilton-cycles for all $x$ even.
\end{lemma}

\begin{proof}
Let $2 K_x$ consist of the point set $\mathbb{Z}_{x-1} \cup \{\infty\}$, and consider the following base Hamilton cycle:

$$\left(\infty, 0, x-2, 1, x-3, 2, x-4, \cdots , \frac{x}{2}, \frac{x-2}{2} \right).$$ If we develop this cycle $\pmod{x-1}$, we will obtain $x-2$ more Hamilton cycles. We must show that these cycles decompose $2 K_x$.

\vspace{5mm}


It is clear that each edge $\{\infty, i\}$ for $i=\{0,1, \cdots, x-2\}$ occurs twice, since we are developing $\{\infty, 0\}$  and $\{\frac{x-2}{2}, \infty\}$ $(\hbox{mod} \ x-1)$. Notice that the differences covered by the base cycle are $\pm 1, \pm 2, \cdots , \pm \frac{x-2}{2}$ in the first half starting from 0, while differences $\pm \frac{x-2}{2}, \cdots, \pm 2, \pm 1 $ are covered in the second half. This shows that every difference occurs twice in the base cycle and therefore, the $x-1$ cycles obtained decompose $2 K_x$ into Hamilton-cycles.
\end{proof}

\

By applying the proof of Lemma \ref{2Kn}, we list the resulting $x-1$ Hamilton cycles that decompose $2 K_x$:
\begin{align*}
C_0 &=\left(\infty, 0, x-2, 1, x-3, 2, x-4, \cdots , \frac{x}{2}, \frac{x-2}{2} \right) \\
C_1 &=\left(\infty, 0+1, x-2+1, 1+1, x-3+1, 2+1, x-4+1, \cdots , \frac{x}{2}+1, \frac{x-2}{2}+1 \right) \\
&\vdots \\
C_{k} &=\left(\infty, 0+k, x-2+k, 1+k, x-3+k, 2+k, x-4+k, \cdots , \frac{x}{2}+k, \frac{x-2}{2}+k \right) \\
&\vdots \\
C_{x-2} &=\left(\infty, 0+(x-2), x-2+(x-2), 1+(x-2), \cdots , \frac{x}{2}+(x-2), \frac{x-2}{2}+(x-2) \right).
\end{align*}

\

Let $C_0,C_1, \ldots, C_{x-2}$ correspond to a set of $x-1$ $MRSM_{\mathbb Z_6}(S,3,x;\Sigma)$. We define an ordering of these matrices as follows. Let $T_0$ be the matrix corresponding to $C_0$, where the columns of $T_0$ are indexed by the edges of $C_0$.

\begin{center}
$T_0=$
\begin{tabular}{ |c|c|c|c|c| } 
\hline
$(\infty,0)$ & $(0,x-2)$ & $(x-2,1)$ & $\hdots$ & $(\frac{x-2}{2},\infty)$ \\  
 \hline
  &  &  &  & \\ 
 \hline
  &  &  &  & \\ 
 \hline
  &  &  &  &\\ 
 \hline
\end{tabular}
\end{center}

\

Then, we can use the following notation to say that $T_0$ corresponds to $C_0$:
$$T_0\rightarrow C_0$$

\

Now, order the remaining $x-2$ matrices as follows:

$$T_j \rightarrow C_0 + j\bigg(\frac{x-2}{2}\bigg)=C_{j(\frac{x-2}{2})}$$

Notice that $T_j=T_{j-1}+\frac{x-2}{2} \pmod{x-1}$ since $T_j \rightarrow C_{j(\frac{x-2}{2})}$, and

\begin{equation} \label{eq1}
\begin{split}
T_{j-1}+\frac{x-2}{2} & \rightarrow C_{(j-1)(\frac{x-2}{2})} + \bigg( \frac{x-2}{2}\bigg) \\
 & = C_0 + (j-1)\bigg(\frac{x-2}{2}\bigg) + \bigg( \frac{x-2}{2}\bigg) \\
 & = C_0 + j \bigg(\frac{x-2}{2}\bigg) -\bigg(\frac{x-2}{2}\bigg) + \bigg( \frac{x-2}{2}\bigg) \\
 & = C_0 + j \bigg(\frac{x-2}{2}\bigg) \\
 & = C_{j(\frac{x-2}{2})}.
\end{split}
\end{equation}

Let $t_{j,l}$ represent the index of column $l$ in matrix $T_j$ for $l= 1,2,\ldots,x$. We define the {\it middle columns} of matrix $T_j$ as the column indices that exclude $t_{j,1}$ and $t_{j,x}$.

\begin{lemma}
\label{Middle columns}
Suppose $x$ is even. For any $T_j \in (T_0, T_1, \ldots, T_{x-2})$, if $l \in \{2,4, \cdots, x-2\}$, then $t_{j,l}=t_{j+1,x-l+1}$.
\end{lemma}

\begin{proof}
We know that each matrix, $T_j$, corresponds to the cycle $C_{j(\frac{x-2}{2})}$ that was constructed in Lemma \ref{2Kn}. Because the cycles were built by using $\mathbb{Z}_{x-1} \cup \{\infty\}$, all calculations are done $\pmod{x-1}$. Without loss of generality, suppose that $C_{j(\frac{x-2}{2})}=(c_{j(\frac{x-2}{2}),1}, c_{j(\frac{x-2}{2}),2}, \ldots, c_{j(\frac{x-2}{2}),x})$, where $c_{j(\frac{x-2}{2}),l}$ represents the $l$-th entry of the $j\bigg(\frac{x-2}{2}\bigg)$-th cycle, and the $l$-th edge of this cycle is defined by

$$\{c_{j(\frac{x-2}{2}),l}, c_{j(\frac{x-2}{2}),l+1}\}= t_{j,l}.$$

Because $T_{j+1}$ corresponds to $C_{(j+1)(\frac{x-2}{2})}$, the $x-l+1$-th column index from $T_{j+1}$, is defined by the edge

$$\{c_{(j+1)(\frac{x-2}{2}),x-l+1}, c_{(j+1)(\frac{x-2}{2}),x-l+2}\}=t_{j+1,x-l+1}.$$

Thus, we want to show that 

\begin{equation}
    c_{j(\frac{x-2}{2}),l}=c_{(j+1)(\frac{x-2}{2}),x-l+1, \ \ \hbox{and}} 
\end{equation}

\begin{equation}
    c_{j(\frac{x-2}{2}),l+1}=c_{(j+1)(\frac{x-2}{2}),x-l+2.}
\end{equation}

Notice that since 
$$C_0 + (j+1)\bigg(\frac{x-2}{2}\bigg)=C_{(j+1)(\frac{x-2}{2})},$$ 
it follows that 
$$c_{(j+1)(\frac{x-2}{2}),x-l+1}= c_{0,x-l+1}+(j+1)\left(\frac{x-2}{2}\right),$$ 
and 
$$c_{(j+1)(\frac{x-2}{2}),x-l+2}= c_{0,x-l+2}+(j+1)\left(\frac{x-2}{2}\right).$$  

 By referring to the base cycle, $C_0$, we have the following. If $k$ is even, then $c_{0,k}=k-\frac{k+2}{2}$ and if $k$ is odd ($k \geq 3$), then $c_{0,k}=x-\frac{k+1}{2}$. By construction of that cycle, observe that if $l \in \{2,4, \cdots, x-2\}$, then we have that the left hand side of (2) is as follows because $k=l$ is even.

\begin{equation}
\begin{split}
c_{j(\frac{x-2}{2}),l} & = l-\left(\frac{(l+2}{2}\right)+j\left(\frac{x-2}{2}\right) \\
 & = \frac{2l-l-2+jx-2j}{2} \\
 & = \frac{l+j-2j-2}{2} \\
 & = \frac{l-j-2}{2}.
\end{split}
\end{equation}

Notice that $x \equiv 1 \pmod{x-1}$, which implies that $2x \equiv 2 \pmod{x-1}$, and consequently $jx \equiv j \pmod{x-1}$.

\

Also, the right hand side of (2) is as follows because $k=x-l+1$ is odd.

\begin{equation}
\begin{split}
c_{(j+1)(\frac{x-2}{2}),x-l+1} &= c_{0,x-l+1}+(j+1)\left(\frac{x-2}{2}\right) \\
 & = x-\left(\frac{(x-l+1)+1}{2}\right)+(j+1)\left(\frac{x-2}{2}\right) \\
 & = \frac{2x-x+l-2+jx-2j+x-2}{2} \\
 & = \frac{(j+2)x+l-2j-4}{2} \\
 & = \frac{j+2+l-2j-4}{2} \\
 & = \frac{l-j-2}{2}.
\end{split}
\end{equation}

Therefore, 
$$c_{j(\frac{x-2}{2}),l}=c_{(j+1)(\frac{x-2}{2}),x-l+1}.$$

So equation (2) is proved.

\

Similarly, the left hand side of (3) is as follows because $k=l+1$ is odd.

\begin{equation}
\begin{split}
c_{j(\frac{x-2}{2}),l+1} & = x-\left(\frac{(l+1)+1}{2}\right)+j\left(\frac{x-2}{2}\right) \\
 & = \frac{2x-l-2+jx-2j}{2} \\
 & = \frac{(j+2)x-l-2j-2}{2} \\
 & = \frac{j+2-l-2j-2}{2} \\
 & = \frac{-l-j}{2}.
\end{split}
\end{equation}

The right hand side of (3) is as follows because $k=x-l+2$ is even.
\begin{equation}
\begin{split}
c_{(j+1)(\frac{x-2}{2}),x-l+2} &= c_{0,x-l+2}+(j+1)\left(\frac{x-2}{2}\right) \\
 & = (x-l+2)-\left(\frac{(x-l+2)+2}{2}\right)+(j+1)\left(\frac{x-2}{2}\right) \\
 & = \frac{2x-2l+4-x+l-4+jx-2j+x-2}{2} \\
 & = \frac{2x+jx-l-2j-2}{2} \\
 & = \frac{(j+2)x-l-2j-2}{2} \\
 & = \frac{j+2-l-2j-2}{2} \\
 & = \frac{-l-j}{2}.
\end{split}
\end{equation}

Therefore, 
$$c_{j(\frac{x-2}{2}),l+1}=c_{(j+1)(\frac{x-2}{2}),x-l+2}.$$

So equation (3) is proved.

\end{proof}

\

Consider the $x-1$ Hamilton cycles that decompose $2K_x$, listed after the proof of Lemma \ref{2Kn} as $C_0, C_1, \ldots, C_{x-2}$. These cycles correspond to the $x-1$ $MRSM$s that we aim to construct, with the columns of each matrix indexed by the edges of the cycle it corresponds to. We order the tables by $(T_0, T_1, \ldots, T_{x-2})$ according to the discussion preceding Lemma \ref{Middle columns}. Then we have the following.

\begin{lemma}
\label{Tables1a}
Let $x \equiv 4,8 \pmod{12}$.  There exists an ordered set of $x-1$ $MRSM_{\mathbb Z_6}(S,3,x;\Sigma)$, $(T_0, T_1, \ldots T_{x-2})$, such that:
\begin{enumerate}
\item For each matrix $T_0,T_1, \ldots, T_{x-3}$, $d \equiv 3 \pmod{6}$ for each $d \in \Sigma$.
\item In $T_{x-2}$, $d \equiv 3 \pmod{6}$ for exactly two $d \in \Sigma$, and $d \equiv 0,2,$ or $4 \pmod{6}$ for one $d \in \Sigma$.
\item Among the set of matrices, there are two columns indexed by $\{u,v\}$, with $\{u,v\} \in E(K_{x})$. The union of these two columns is $Z_6$.
\end{enumerate}
\end{lemma}

\begin{proof}
Let $2K_x$ consists of the point set $\mathbb{Z}_{x-1} \cup \{\infty\}$ and consider the base Hamilton cycle $C_0= \left(\infty, 0, x-2, 1, x-3, 2, x-4, \cdots , \frac{x}{2}, \frac{x-2}{2} \right)$. By Lemma \ref{2Kn}, $2K_x$ can be decomposed into $x$-cycles by developing this Hamilton cycle$\pmod{x-1}$ producing the cycles $C_0, C_1, \ldots, C_{x-2}$. As described in the discussion before Lemma \ref{Middle columns}, we will construct the $x-1$ matrices so that $T_j$ corresponds to cycle $C_{j(\frac{x-2}{2})}$, for $j=0,1,2, \ldots,x-2$. The columns of matrix $T_j$ are indexed by the edges of cycle $C_{j(\frac{x-2}{2})}$. Therefore, among the set of matrices, we see every pair $\{u,v\}$ exactly twice among the set of column indices. We now describe how to fill the columns for each table. We will not pay attention to the way the columns are indexed, and simply describe the way the columns of each table are filled. Therefore we will refer to the columns as column 1, column 2, ..., column $x$.

\begin{center}
    {\bf Algorithm for filling the $x-1$ matrices $T_j$:}
\end{center}

\begin{itemize}
    \item Matrix $T_0$
    \begin{enumerate}
        \item Fill column 1 with $(0,2,5)^T$.
        \item Fill column $i$, $i\in \{2, 4,\ldots x-2\}$, with $(0,1,2)^T$.
        \item Fill column $i$, $i\in \{3, 5,\ldots x-1\}$, with $(0,2,1)^T$.
        \item Fill column $x$ with $(3,4,1)^T$.
    \end{enumerate}
    \item Matrices $T_k$, $k \in \{1,2, \dots x-3 \}$.
    \begin{enumerate}
        \item Fill column 1 in $T_k$ with $(a+3,b+3,c+3)^T$, where $(a,b,c)^T$ is column $x$ from $T_{k-1}$.
        \item Fill column $i$ in $T_k$, $i\in \{3, 5,\ldots x-1\}$ with $(a+3,b+3,c+3)^T$, where $(a,b,c)^T$ is column ${x+1-i}$ from $T_{k-1}$.
        \item Fill column $i$ in $T_k$, $i\in \{2, 4,\ldots x-2\}$, with $(a,b,c)^T$, where $(a,b,c)^T$ is column $x+1-i$ from $T_{k-1}$.
        \item Fill column $x$ with $(a,b,c)^T$, where $(a,b,c)^T$ is column 1 from $T_{k-1}$.
    \end{enumerate}
    \item Matrix $T_{x-2}$.
    \begin{enumerate}
       \item Fill column 1 in $T_{x-2}$ with $(a+3,b,c)^T$, where $(a,b,c)^T$ is column $x$ from $T_{x-3}$.
        \item Fill column $i$ in $T_{x-2}$, $i\in \{3, 5,\ldots x-1\}$ with $(a+3,b+3,c+3)^T$, where $(a,b,c)^T$ is column $x+1-i$ from $T_{x-3}$.
        \item Fill column $i$ in $T_{x-2}$, $i\in \{2, 4,\ldots x-2\}$ with $(a,b,c)^T$, where $(a,b,c)^T$ is column $x+1-i$ from $T_{x-3}$.
        \item Fill column $x$ in $T_{x-2}$ with $(a+3,b+4,c+4)^T$, where $(a,b,c)^T$ is column 1 from $T_{x-3}$.
        \item Replace the entries of any two consecutive columns $i$ and $(i+1)$ in $T_{x-2}$, where
        
        $i\in \{2, 4, 6,\ldots x-2\}$, as follows:
        $$\hbox{Column} \ i: (a,b,c)^T \rightarrow (a+1,b+1,c-2)^T$$
        $$\hbox{Column} \ i+1: (a,b,c)^T \rightarrow (a+1,b-2,c+1)^T$$
    \end{enumerate}
\end{itemize}

\vspace{5mm}
{\bf Proof of item 1.}

\vspace{0.5mm}

First we will prove that each row of $T_0$ satisfies the condition. Then, we will prove by induction on $x$ that each row of the next $x-3$ matrices also satisfy the condition.

\

Note that in the first row of $T_0$, columns 1 and $x$ are filled with $(0,2,5)^T$ and $(3,4,1)^T$ respectively. The reminder of the first row is filled with zeros, so the sum is 3. In the second row, the first entry was filled with 2, the last entry was filled with 4; and half of the middle columns are filled with 1, while the other half are filled with 2. Since $\frac{x-2}{2}$ is odd when $x \equiv 4,8 \pmod{12}$, we are adding $1+2=3$ an odd number of times. Thus, the sum of the second row is $d\equiv 3 \ \ \hbox{(mod 6)}$. Similarly for the third row, the first entry was filled with 5, the last entry was filled with 1,  and the middle entries alternate between 1 and 2. Thus, the sum of the entries of the third row is $d\equiv 3 \pmod{6}$.

\
    
We wish to show that the entries in each row of matrices $T_1, T_2, \ldots T_{x-3}$ sum to $d\equiv 3 \pmod{6}$. We will prove this by induction on the number of matrices.
    
\

For the base case, we will show that $T_1$ satisfies the condition in item 1. We have already shown that $T_0$ satisfies the condition, and by the algorithm, half of the middle columns (an odd number) from $T_0$ are repeated in $T_1$; and the other half are given by adding 3 to the same set of columns from $T_0$ in a different order. Column 1 of $T_1$ is obtained by adding $3$ to column $x$ from $T_0$, and column $x$ of $T_1$ is equals to column 1 from $T_0$. Thus we have added 3 to the columns of $T_0$ an even number of times, so the sum of each row of $T_1$ is congruent to the sum of each row of $T_0 \pmod{6}$. Thus $T_1$, satisfies the condition.

Assuming $T_{k-1}$ satisfies the condition, then by the same argument, any matrix \
$T_k$, $k \in \{2, \ldots, x-3\}$, also satisfies the condition. Therefore, item 1 holds, by induction.
    
\

{\bf Proof of item 2.}

Let $R_k^j$ represent the sum of row $k$ in matrix $j$. According to the algorithm, at the end of step 4, we have $R_1^{x-2} \equiv R_1^{x-3} +3 \pmod{6}$, $R_2^{x-2} \equiv R_2^{x-3} +1 \pmod{6}$, and $R_3^{x-2} \equiv R_3^{x-3} +1 \pmod{6}$. Then after step 5, we get $R_1^{x-2} \equiv R_1^{x-3} +5 \pmod{6} \equiv 2 \pmod{6}$, $R_2^{x-2} \equiv R_2^{x-3} \equiv 3 \pmod{6}$, and $R_3^{x-2} \equiv R_3^{x-3} \equiv 3 \pmod{6}$. Therefore item 2 holds.

\

{\bf Proof of item 3.}
    
We now refer to the way the columns of each table are indexed. As in Lemma \ref{Middle columns}, we define $t_{j,l}$ to be the index of column $l$ in table $T_j$ for $l \in \{1,2, \ldots, x\}$ and $j \in \{0,1, \ldots, x-2\}$, where all calculations on $j$ are done $\pmod{x-1}$.

Suppose $t_{j,l}$ is a middle column, and $t_{j,l}=(u,v)$. Then it is clear that $u,v \neq \infty$. According to the algorithm, the entries in column $t_{j,l} \in \{(0,1,2)^T, (0,2,1)^T, (3,4,5)^T, (3,5,4)^T\}$. Also, if the entries in column $t_{j,l}$ are $(a,b,c)^T$, then it follows from the algorithm that the entries in column $t_{j+1,x+1-l}$ are $(a+3, b+3, c+3)^T$. By Lemma \ref{Middle columns}, we have that $t_{j,l}=t_{j+1, x+1-l}$; therefore, the union of the entries in the columns indexed by $(u,v)$ is $\{0,1,2,3,4,5\}$.

When $j=x-2$, we must also consider the entries in the columns affected by step 5 of the algorithm. So consider any pair of consecutive columns of $T_j$; $t_{j,l}$ and $t_{j,l+1}$, $l \in \{2,4,6, \ldots, x-2\}$. Because $x-2 \equiv 2 \ \hbox{or } 6 \pmod{12}$, it follows that $x-2 \equiv 2 \pmod{4}$, which implies $\frac{x-2}{2}$ is odd. Therefore, according to the algorithm, the entries in column $t_{j,l}$ and $t_{j,l+1}$ before step 5 are $(3,4,5)^T$ and $(3,5,4)^T$ respectively, since we added 3 to $(0,1,2)^T$ and $(0,2,1)^T$ an odd number of times. However, after step 5, the entries of these columns become $(4,5,3)^T$ and $(4,3,5)^T$ respectively. So while the order of the entries change, the set of entries remains the same. Thus the union of the entries in the columns indexed by $(u,v)$ in $T_{x-2}$ is again $\{0,1,2,3,4,5\}$.

\

We now turn to the case when $u=\infty$ or $v=\infty$. For this case, we focus on columns from the matrices that are on the ends, i.e, $t_{j,1}$ and $t_{j,x}$. For any $j \in \{0, \ldots,x-2\}$, we have that $T_j \rightarrow C_{j\big(\frac{x-2}{2} \big)}$, so $t_{j,x}=((j+1) \big(\frac{x-2}{2}\big), \infty)$. Also, since $T_{j+1} \rightarrow C_{(j+1) \big(\frac{x-2}{2} \big)}$, we have $t_{(j+1),1}=(\infty, (j+1) \big(\frac{x-2}{2}\big))$. Recall that if $(u,v)$ is a column index in some $MRSM$ $T$, then any entry, $d$, in column $(u,v)$ represents the cross difference, $d$, from $G_u$ to $G_v$. If $(u,v)$ is a column indexed in some difference table $T'$, then any entry, $d'$, in column $(v,u)$ represents the cross difference $d$ from $G_v$ to $G_u$. However, we can also represent this difference as the cross difference $-d'$ from $G_u$ to $G_v$. Therefore, the entries in column $t_{j+1,1}$ represent cross differences from $G_{\infty}$ to $G_{(j+1) \big(\frac{x-2}{2} \big)}$ or equivalently, the negatives of these entries represent cross differences from $G_{(j+1) \big(\frac{x-2}{2} \big)}$ to $G_{\infty}$.

By the algorithm, the entries in column $t_{j,x}$ are from the set $\{(3,4,1)^T, (0,2,5)^T, (0,1,4)^T, \\ (3,5,2)^T\}$, and if column $t_{j,x}$ contains entries $(a,b,c)^T$, then column $t_{j+1,1}$ contains entries \\ $(a+3,b+3,c+3)^T$. The following table gives the possibilities of the entries in column \\ $((j+1) \big(\frac{x-2}{2}\big), \infty))$, the corresponding entries in column $(\infty, (j+1) \big(\frac{x-2}{2}\big))$ given by the algorithm, and the negatives of the cross differences listed in column $(\infty, (j+1) \big(\frac{x-2}{2}\big))$.

\
\begin{center}
\begin{tabularx}{0.8\textwidth} { 
  | >{\centering\arraybackslash}X 
  | >{\centering\arraybackslash}X 
  | >{\centering\arraybackslash}X | }
 \hline
 Entries in column $((j+1) \big(\frac{x-2}{2}\big), \infty))$  & Entries in column $(\infty, (j+1) \big(\frac{x-2}{2}\big))$ & Negatives of column $(\infty, (j+1) \big(\frac{x-2}{2}\big))$ \\
 \hline
 3  & 0  & 0  \\
 4  & 1  & 5  \\
 1  & 4  & 2  \\
\hline
 0  & 3  & 3  \\
 2  & 5  & 1  \\
 5  & 2  & 4  \\
\hline
 0  & 3  & 3  \\
 1  & 4  & 2  \\
 4  & 1  & 5  \\
\hline
 3  & 0  & 0  \\
 5  & 2  & 4  \\
 2  & 5  & 1  \\
 \hline
\end{tabularx}
\end{center}
\

In each row of the table, the union of the first and third entries gives the cross differences covered by the edge $((j+1) \big(\frac{x-2}{2}\big), \infty))$. Notice that in each row, the union is $\{0,1,2,3,4,5\}$. Thus the union of the entries in the columns indexed by $(u,v)$, where $u=\infty$ or $v=\infty$, is $\{0,1,2,3,4,5\}$.
\end{proof}

\

\begin{lemma}
\label{Tables1b}
Let $x \equiv 2,10 \pmod{12}$.  There exists an ordered set of $x-1$ $MRSM_{\mathbb Z_6}(S,3,x;\Sigma)$, $(T_0, T_1, \ldots T_{x-2})$, such that:
\begin{enumerate}
\item In $T_{0}$, $d \equiv 3 \pmod{6}$ for exactly two $d \in \Sigma$, and $d=0$ for one $d \in \Sigma$. 
\item In $T_1, \ldots, T_{x-3}$, $d \equiv 3 \pmod{6}$ for each $d \in \Sigma$.
\item In $T_{x-2}$, $d \equiv 3 \pmod{6}$ for exactly two $d \in \Sigma$, and $d \equiv 0,2,$ or $4 \pmod{6}$ for one $d \in \Sigma$.
\item Among the set of matrices, there are two columns indexed by $\{u,v\}$, with $\{u,v\} \in E(K_{x})$. The union of these two columns is $Z_6$.
\end{enumerate}
\end{lemma}
\begin{proof}
Following a similar idea as in the previous case, proved in Lemma \ref{Tables1a}, we consider the base Hamilton cycle $C_0= \left(\infty, 0, x-2, 1, x-3, 2, x-4, \cdots , \frac{x}{2}, \frac{x-2}{2} \right)$. By Lemma \ref{2Kn}, $2K_x$ can be decomposed into $x$-cycles by developing this Hamilton cycle$\pmod{x-1}$ producing the cycles $C_0, C_1, \ldots, C_{x-2}$. As described in the discussion before Lemma \ref{Middle columns}, we will construct the $x-1$ matrices so that matrix $T_j$ correspond to cycle $C_{j(\frac{x-2}{2})}$, for $j=0,1,2, \ldots,x-2$. The columns of table $T_j$ are indexed by the edges of cycle $C_{j(\frac{x-2}{2})}$. Therefore, among the set of matrices, we see every pair $\{u,v\}$ exactly twice among the set of column indices. We now describe how to fill in the columns of each $T$, when $x \equiv 2,10 \pmod{12}$.

\begin{center}
    {\bf Algorithm for filling the $x-1$ matrices $T_j$:}
\end{center}

\begin{itemize}
    \item Table $T_0$
    \begin{enumerate}
        \item Fill column 1 with $(0,2,5)^T$.
        \item Fill column $i$, $i\in \{2, 4,\ldots x-2\}$ with $(0,1,2)^T$.
        \item Fill column $i$, $i\in \{3, 5,\ldots x-1\}$ with $(0,2,1)^T$.
        \item Fill column $x$ with $(0,1,4)^T$.
    \end{enumerate}
     \item Table $T_1$.
    \begin{enumerate}
        \item Fill column 1 in $T_1$ with $(a+3,b+3,c+3)^T$, where $(a,b,c)^T$ is column $x$ from $T_0$.
        \item Fill column $i$ in $T_1$, $i\in \{3, 5,\ldots x-1\}$ with $(a+3,b+3,c+3)^T$, where $(a,b,c)^T$ is column $x+1-i$ from $T_0$.
        \item Fill column $i$ in $T_1$, $i\in \{2, 4,\ldots x-2\}$ with $(a,b,c)^T$, where $(a,b,c)^T$ is column $x+1-i$ from $T_0$.
        \item Fill column $x$ in $T_1$ with $(a+0,b+3,c+3)^T$, where $(a,b,c)^T$ is column 1 from $T_0$.
    \end{enumerate}
    \item Tables $T_k$, $k \in \{2, \dots x-3 \}$.
    \begin{enumerate}
        \item Fill column 1 in $T_k$ with $(a+3,b+3,c+3)^T$, where $(a,b,c)^T$ is column $x$ in $T_{k-1}$.
        \item Fill column $i$ in $T_k$, $i\in \{3, 5,\ldots x-1\}$ with $(a+3,b+3,c+3)^T$, where $(a,b,c)^T$ is column $x+1-i$ from $T_{k-1}$.
        \item Fill column $i$ in $T_k$, $i\in \{2, 4,\ldots x-2\}$ with $(a,b,c)^T$, where $(a,b,c)^T$ is column $x+1-i$ from $T_{k-1}$.
        \item Fill column $x$ with $(a+3,b+3,c+3)^T$, where $(a,b,c)^T$ is column 1 from $T_{k-1}$.
    \end{enumerate}
    \item Table $T_{x-2}$.
    \begin{enumerate}
       \item Fill column 1 in $T_{x-2}$ with $(a+3,b,c)^T$, where $(a,b,c)^T$ is column $x$ from $T_{x-3}$.
        \item Fill column $i$ in $T_{x-2}$, $i\in \{3, 5,\ldots x-1\}$ with $(a+3,b+3,c+3)^T$, where $(a,b,c)^T$ is column $x+1-i$ from $T_{x-3}$.
        \item Fill column $i$ in $T_{x-2}$, $i\in \{2, 4,\ldots x-2\}$ with $(a,b,c)^T$, where $(a,b,c)^T$ is column $x+1-i$ from $T_{x-3}$.
        \item Fill column $x$ in $T_{x-2}$ with $(a+0,b+1,c+1)^T$, where $(a,b,c)^T$ is column 1 from $T_{x-3}$.
        \item Replace the entries of any two consecutive columns $i$ and $i+1$ in $T_{x-2}$, where $i\in \{2, 4, 6,\ldots x-2\}$, as follows:
        $$\hbox{Column} \ i: (a,b,c)^T \rightarrow (a+1,b+1,c-2)^T$$
        $$\hbox{Column} \ i+1: (a,b,c)^T \rightarrow (a+1,b-2,c+1)^T$$
    \end{enumerate}
\end{itemize}

\vspace{5mm}
{\bf Proof of item 1.}

It is clear that the first row of $T_0$ contains all 0's, so the sum is 0. In the second row, the first entry was filled with 2, the last entry was filled with 1; and half of the middle columns are filled with 1, while the other half are filled with 2. Since $\frac{x-2}{2}$ is even when $x \equiv 2,10 \pmod{12}$, we are adding $1+2=3$ an odd number of times. Thus, the sum of the second row is $d\equiv 3 \ \ \hbox{(mod 6)}$. Similarly for the third row, the first entry was filled with 5, the last entry was filled with 4,  and the middle entries alternate between 1 and 2. Thus, the sum of the entries of the third row is $d\equiv 3 \pmod{6}$. Thus item 1 holds.

\

{\bf Proof of item 2.}
According to the algorithm, $R_1^1 \equiv R_1^0 +3 \pmod{6} \equiv 3 \pmod{6}$, $R_2^1 \equiv R_2^0 +0 \pmod{6} \equiv 3 \pmod{6}$, and $R_3^1 \equiv R_3^0 +0 \pmod{6} \equiv 3 \pmod{6}$. Thus, $T_1$ satisfies the condition given in item 2.

\

For $k \in \{2, \ldots,x-3\}$, we use induction, starting with the base case, $T_2$. We have already shown that $T_1$ satisfies the condition. By the algorithm, half of the middle columns (an even number) from $T_1$ are repeated in $T_2$, and the other half are obtained by adding 3 to the same set of columns from $T_1$ in a different order. Column 1 of $T_2$ is obtained by adding $3$ to column $x$ from $T_1$, and column $x$ of $T_2$ is obtained by adding $3$ to column 1 from $T_1$. Thus we have added 3 to the columns of $T_1$ an even number of times. Therefore, the sum of each row of $T_2$ is congruent to the sum of each row of $T_1 \pmod{6}$. Thus $T_2$, satisfies the condition. Assuming $T_{k-1}$ satisfies the condition, then by the same argument, any table $T_k$, $k \in \{3, \ldots, x-3\}$, also satisfies the condition. Therefore, item 2 holds, by induction.

\
    
{\bf Proof of item 3.}
    
According to the algorithm, at the end of step 4, we have $R_1^{x-2} \equiv R_1^{x-3} +3 \pmod{6}$, $R_2^{x-2} \equiv R_2^{x-3} +1 \pmod{6}$, and $R_3^{x-2} \equiv R_3^{x-3} +1 \pmod{6}$. Then after step 5, we get $R_1^{x-2} \equiv R_1^{x-3} +5 \pmod{6} \equiv 2 \pmod{6}$, $R_2^{x-2} \equiv R_2^{x-3} \equiv 3 \pmod{6}$, and $R_3^{x-2} \equiv R_3^{x-3} \equiv 3 \pmod{6}$. Therefore item 3 holds.
    
\

{\bf Proof of item 4.}    

As in the proof of Lemma \ref{Tables1a}, we first consider middle columns. Consider any column $t_{j,l}$, $l \in \{2,4,6, \ldots, x-2\}$, and suppose $t_{j,l}=(u,v)$. Then we have $u,v \neq \infty$. According to the algorithm, the entries in column $t_{j,l} \in \{(0,1,2)^T, (0,2,1)^T, (3,4,5)^T, (3,5,4)^T\}$. Also, if the entries in column $t_{j,l}$ are $(a,b,c)^T$, then it follows from the algorithm that the entries in column $t_{j+1,x+1-l}$ are $(a+3, b+3, c+3)^T$. By Lemma \ref{Middle columns}, we have that $t_{j,l}=t_{j+1, x+1-l}$; therefore, the union of the entries in the columns indexed by $(u,v)$ is $\{0,1,2,3,4,5\}$.

When $j=x-2$, we must also consider the entries in the columns affected by step 5 of the algorithm. So consider any pair of consecutive columns of $T_j$; $t_{j,l}$ and $t_{j,l+1}$, $l \in \{2,4,6, \ldots, x-2\}$. Because $x-2 \equiv 0 \ \hbox{or } 8 \pmod{12}$, it follows that $x-2 \equiv 0 \pmod{4}$, which implies $\frac{x-2}{2}$ is even. Therefore, according to the algorithm, the entries in column $t_{j,l}$ and $t_{j,l+1}$ before step 5 are $(0,1,2)^T$ and $(0,2,1)^T$ respectively. However, after step 5, the entries become  $(1,2,0)^T$ and $(1,0,2)^T$ respectively, since we added 3 to $(0,1,2)^T$ and $(0,2,1)^T$ an even number of times. So while the order of the entries change, the set of entries remains the same. Thus the union of the entries in the columns indexed by $(u,v)$ in $T_{x-2}$ is again $\{0,1,2,3,4,5\}$ for $u,v \neq 0$.

\

We now turn to the case when $u=\infty$ or $v=\infty$. For this case, we focus on columns from the matrices that are on the ends, i.e, $t_{j,1}$ and $t_{j,x}$. For any $j \in \{0, \ldots,x-2\}$, we have that $T_j \rightarrow C_{j\big(\frac{x-2}{2} \big)}$, so $t_{j,x}=((j+1) \big(\frac{x-2}{2}\big), \infty)$. Also, since $T_{j+1} \rightarrow C_{(j+1) \big(\frac{x-2}{2} \big)}$, we have $t_{(j+1),1}=(\infty, (j+1) \big(\frac{x-2}{2}\big))$. 
    
As in the proof of Lemma \ref{Tables1a}, the entries in column $t_{j+1,1}$ represent cross differences from $G_{\infty}$ to $G_{(j+1) \big(\frac{x-2}{2} \big)}$ or equivalently, the negatives of these entries represent cross differences from $G_{(j+1) \big(\frac{x-2}{2} \big)}$ to $G_{\infty}$.

By the algorithm, the entries in column $t_{j,x}$ are from the set $\{(0,1,4)^T, (0,5,2)^T, (3,5,2)^T\}$, and if column $t_{j,x}$ contains entries $(a,b,c)^T$, then column $t_{j+1,1}$ contains entries $(a+3,b+3,c+3)^T$. The following table gives the possibilities of the entries in column $((j+1) \big(\frac{x-2}{2}\big), \infty))$, the corresponding entries in column $(\infty, (j+1) \big(\frac{x-2}{2}\big))$ given by the algorithm, and the negatives of the cross differences listed in column $(\infty, (j+1) \big(\frac{x-2}{2}\big))$.

\
\begin{center}
\begin{tabularx}{0.8\textwidth} { 
  | >{\centering\arraybackslash}X 
  | >{\centering\arraybackslash}X 
  | >{\centering\arraybackslash}X | }
 \hline
 Entries in column $((j+1) \big(\frac{x-2}{2}\big), \infty))$  & Entries in column $(\infty, (j+1) \big(\frac{x-2}{2}\big))$ & Negatives of column $(\infty, (j+1) \big(\frac{x-2}{2}\big))$ \\
 \hline
 0  & 3  & 3  \\
 1  & 4  & 2  \\
 4  & 1  & 5  \\
\hline
 0  & 3  & 3  \\
 5  & 2  & 4  \\
 2  & 5  & 1  \\
\hline
 3  & 0  & 0  \\
 5  & 2  & 4  \\
 2  & 5  & 1  \\
\hline
\end{tabularx}
\end{center}
\

In each row of the table, the union of the first and third entries gives the cross differences covered by the edge $((j+1) \big(\frac{x-2}{2}\big), \infty))$. Notice that in each row, the union is $\{0,1,2,3,4,5\}$. Thus the union of the entries in the columns indexed by $\{u,v\}$ is $\{0,1,2,3,4,5\}$.

\end{proof}

\section{Finding Decompositions}

In this section, we utilize the methods described previously to solve the case when $\gcd(6,N)=2$ and $N=2x$.

\subsection{$\gcd(6,N)=2$ and $x \equiv 2 \ \hbox{or} \ 4 \pmod{6}$}

\begin{theorem}
\label{x=2}
 There exists a solution to ${\hbox{HWP}}(12t;6,4;1,6t-2)$ for all $t \geq 1$.
\end{theorem}

\begin{proof}
When $t=1$, we prove that there exists a solution to ${\hbox{HWP}}(12;6,4;1,4)$. 

\

Let $G_0$ and $G_1$ be the two parts of $K_{(6:2)}$, and consider the 2-cycle $H=(0,1)$. Let $d_1=1$ and $d_2=2$, so $d=d_1+d_2 \equiv 3 \pmod{6}$. Since $\gcd(3,6)=3$, it follows from Lemma \ref{Lemma2} that $G[\Delta]$ consists of 3 components, and each component is a 4-cycle. Denote this 2-factor as $A$. Following in this way, we can also let $d_1=0$ and $d_2=3$ to obtain another 4-cycle factor, $B$. At this point we have covered cross differences , $0,1,3,4$ from $G_0$ to $G_1$, and cross differences $0,2,3,5$ from $G_1$ to $G_0$. The cross difference 2 from $G_0$ to $G_1$ is equivalent to cross difference 4 from $G_1$ to $G_0$. This is covered in $C$ with $d_1=2$ and $d_2=4$. The cross difference 5 from $G_0$ to $G_1$ (equivalently, 1 from $G_1$ to $G_0$) is covered by $D$. One can easily check that the inside edges in $C \cup D \cup E \cup F$ produce $2K_6$. 

Thus $A \cup B \cup C \cup D \cup E \cup F$ gives the solution. The solution is illustrated in Figure \ref{12;6,4;1,4}.

\

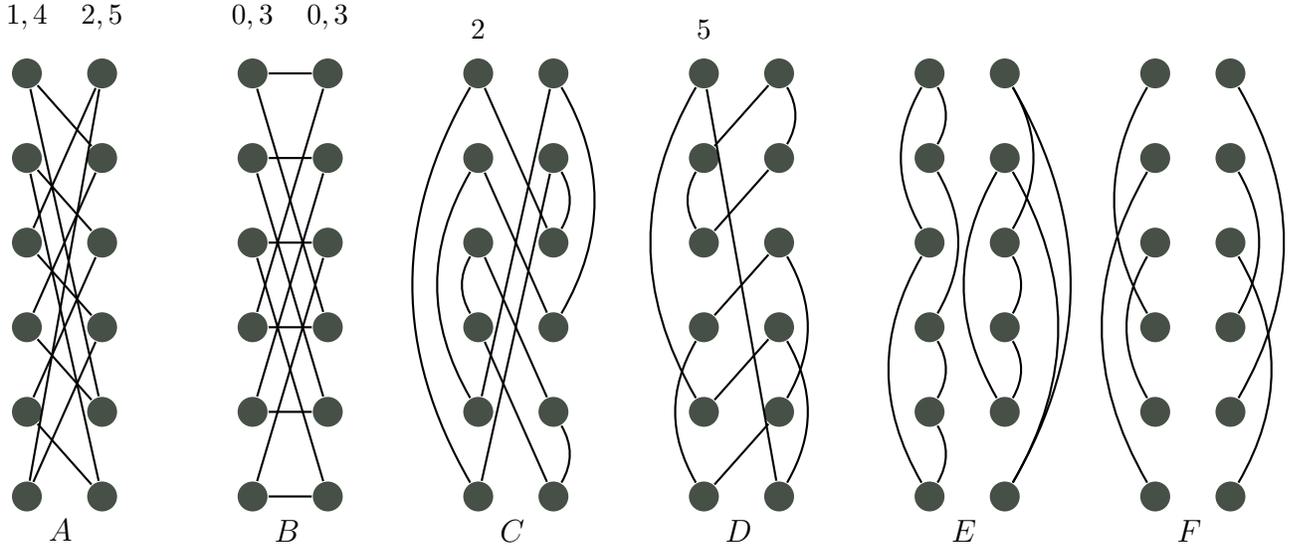
\begin{figure}[h!]
\centering
\definecolor{mygray}{RGB}{70,80,70}

\begin{tikzpicture}[thick,
  every node/.style={circle},
  fsnode/.style={fill=mygray},
  gsnode/.style={fill=mygray},
  hsnode/.style={fill=mygray},
  isnode/.style={fill=mygray}
]

\begin{scope}[start chain=going below,node distance=7mm]
  \node[fsnode,on chain] (f1) [label={above:\small ${1,4}$}] {};
  \node[fsnode,on chain] (f2) {};
  \node[fsnode,on chain] (f3) {};
  \node[fsnode,on chain] (f4) {};
  \node[fsnode,on chain] (f5) {};
  \node[fsnode,on chain] (f6) [label={below right:${A}$}]{};
\end{scope}

\begin{scope}[xshift=1cm,yshift=0cm,start chain=going below,node distance=7mm]
  \node[gsnode,on chain] (g1) [label={above:\small ${2,5}$}] {};
  \node[gsnode,on chain] (g2) {};
  \node[gsnode,on chain] (g3) {};
  \node[gsnode,on chain] (g4) {};
  \node[gsnode,on chain] (g5) {};
  \node[gsnode,on chain] (g6) {};
\end{scope}

\begin{scope}[xshift=3cm,yshift=0cm,start chain=going below,node distance=7mm]
  \node[hsnode,on chain] (h1) [label={above:\small ${0,3}$}]{};
  \node[hsnode,on chain] (h2) {};
  \node[hsnode,on chain] (h3) {};
  \node[hsnode,on chain] (h4) {};
  \node[hsnode,on chain] (h5) {};
  \node[hsnode,on chain] (h6) [label={below right:${B}$}]{};
\end{scope}

\begin{scope}[xshift=4cm,yshift=0cm,start chain=going below,node distance=7mm]
  \node[isnode,on chain] (i1) [label={above:\small ${0,3}$}] {};
  \node[isnode,on chain] (i2) {};
  \node[isnode,on chain] (i3) {};
  \node[isnode,on chain] (i4) {};
  \node[isnode,on chain] (i5) {};
  \node[isnode,on chain] (i6) {};
\end{scope}

\begin{scope}[xshift=6cm,yshift=0cm,start chain=going below,node distance=7mm]
  \node[hsnode,on chain] (j1) [label={above:\small ${2}$}]{};
  \node[hsnode,on chain] (j2) {};
  \node[hsnode,on chain] (j3) {};
  \node[hsnode,on chain] (j4) {};
  \node[hsnode,on chain] (j5) {};
  \node[hsnode,on chain] (j6) [label={below right:${C}$}] {};
\end{scope}

\begin{scope}[xshift=7cm,yshift=0cm,start chain=going below,node distance=7mm]
  \node[isnode,on chain] (k1) {};
  \node[isnode,on chain] (k2) {};
  \node[isnode,on chain] (k3) {};
  \node[isnode,on chain] (k4) {};
  \node[isnode,on chain] (k5) {};
  \node[isnode,on chain] (k6) {};
\end{scope}

\begin{scope}[xshift=9cm,yshift=0cm,start chain=going below,node distance=7mm]
  \node[hsnode,on chain] (l1) [label={above:\small ${5}$}]{};
  \node[hsnode,on chain] (l2) {};
  \node[hsnode,on chain] (l3) {};
  \node[hsnode,on chain] (l4) {};
  \node[hsnode,on chain] (l5) {};
  \node[hsnode,on chain] (l6) [label={below right:${D}$}] {};
\end{scope}

\begin{scope}[xshift=10cm,yshift=0cm,start chain=going below,node distance=7mm]
  \node[isnode,on chain] (m1) {};
  \node[isnode,on chain] (m2) {};
  \node[isnode,on chain] (m3) {};
  \node[isnode,on chain] (m4) {};
  \node[isnode,on chain] (m5) {};
  \node[isnode,on chain] (m6) {};
\end{scope}

\begin{scope}[xshift=12cm,yshift=0cm,start chain=going below,node distance=7mm]
  \node[hsnode,on chain] (n1) {};
  \node[hsnode,on chain] (n2) {};
  \node[hsnode,on chain] (n3) {};
  \node[hsnode,on chain] (n4) {};
  \node[hsnode,on chain] (n5) {};
  \node[hsnode,on chain] (n6) [label={below right:${E}$}] {};
\end{scope}

\begin{scope}[xshift=13cm,yshift=0cm,start chain=going below,node distance=7mm]
  \node[isnode,on chain] (o1) {};
  \node[isnode,on chain] (o2) {};
  \node[isnode,on chain] (o3) {};
  \node[isnode,on chain] (o4) {};
  \node[isnode,on chain] (o5) {};
  \node[isnode,on chain] (o6) {};
\end{scope}

\begin{scope}[xshift=15cm,yshift=0cm,start chain=going below,node distance=7mm]
  \node[hsnode,on chain] (p1) {};
  \node[hsnode,on chain] (p2) {};
  \node[hsnode,on chain] (p3) {};
  \node[hsnode,on chain] (p4) {};
  \node[hsnode,on chain] (p5) {};
  \node[hsnode,on chain] (p6) [label={below right:${F}$}] {};
\end{scope}

\begin{scope}[xshift=16cm,yshift=0cm,start chain=going below,node distance=7mm]
  \node[isnode,on chain] (q1) {};
  \node[isnode,on chain] (q2) {};
  \node[isnode,on chain] (q3) {};
  \node[isnode,on chain] (q4) {};
  \node[isnode,on chain] (q5) {};
  \node[isnode,on chain] (q6) {};
\end{scope}

\draw (f1) -- (g2);
\draw (g2) -- (f4);
\draw (f4) -- (g5);
\draw (g5) -- (f1);
\draw (f2) -- (g3);
\draw (g3) -- (f5);
\draw (f5) -- (g6);
\draw (g6) -- (f2);
\draw (f3) -- (g4);
\draw (g4) -- (f6);
\draw (f6) -- (g1);
\draw (g1) -- (f3);
\draw (h1) -- (i1);
\draw (i1) -- (h4);
\draw (h4) -- (i4);
\draw (i4) -- (h1);
\draw (h2) -- (i2);
\draw (i2) -- (h5);
\draw (h5) -- (i5);
\draw (i5) -- (h2);
\draw (h3) -- (i3);
\draw (i3) -- (h6);
\draw (h6) -- (i6);
\draw (i6) -- (h3);
\draw (j1) -- (k3);
\draw (k3) edge[bend right] node [left] {} node {} (k2);
\draw (k2) -- (j6);
\draw (j6) edge[bend left] node [left] {} node {}  (j1);
\draw (j2) -- (k4);
\draw (k4) edge[bend right] node [left] {} node {}  (k1);
\draw (k1) -- (j5);
\draw (j5) edge[bend left] node [left] {} node {}  (j2);
\draw (j3) -- (k5);
\draw (k5) edge[bend left] node [left] {} node {}  (k6);
\draw (k6) -- (j4);
\draw (j4) edge[bend left] node [left] {} node {}  (j3);
\draw (l1) -- (m6);
\draw (m6) edge[bend right] node [left] {} node {} (m4);
\draw (m4) -- (l5);
\draw (l5) edge[bend left] node [left] {} node {}  (l1);
\draw (l2) -- (m1);
\draw (m1) edge[bend left] node [left] {} node {}  (m2);
\draw (m2) -- (l3);
\draw (l3) edge[bend left] node [left] {} node {}  (l2);
\draw (l4) -- (m3);
\draw (m3) edge[bend left] node [left] {} node {}  (m5);
\draw (m5) -- (l6);
\draw (l6) edge[bend left] node [left] {} node {}  (l4);
\draw (n1) edge[bend left] node [left] {} node {}  (n2);
\draw (n2) edge[bend left] node [left] {} node {}  (n4);
\draw (n4) edge[bend left] node [left] {} node {}  (n5);
\draw (n5) edge[bend left] node [left] {} node {}  (n6);
\draw (n6) edge[bend left] node [left] {} node {}  (n3);
\draw (n3) edge[bend left] node [left] {} node {}  (n1);
\draw (o1) edge[bend left] node [left] {} node {}  (o3);
\draw (o3) edge[bend left] node [left] {} node {}  (o4);
\draw (o4) edge[bend left] node [left] {} node {}  (o5);
\draw (o5) edge[bend left] node [left] {} node {}  (o2);
\draw (o2) edge[bend left] node [left] {} node {}  (o6);
\draw (o6) edge[bend right] node [left] {} node {}  (o1);
\draw (p1) edge[bend right] node [left] {} node {}  (p4);
\draw (p2) edge[bend right] node [left] {} node {}  (p6);
\draw (p3) edge[bend right] node [left] {} node {}  (p5);
\draw (q1) edge[bend left] node [left] {} node {}  (q5);
\draw (q2) edge[bend left] node [left] {} node {}  (q4);
\draw (q3) edge[bend left] node [left] {} node {}  (q6);
\end{tikzpicture}
\caption{$\hbox{HWP}(12;6,4;1,4)$}
\label{12;6,4;1,4}
\end{figure}

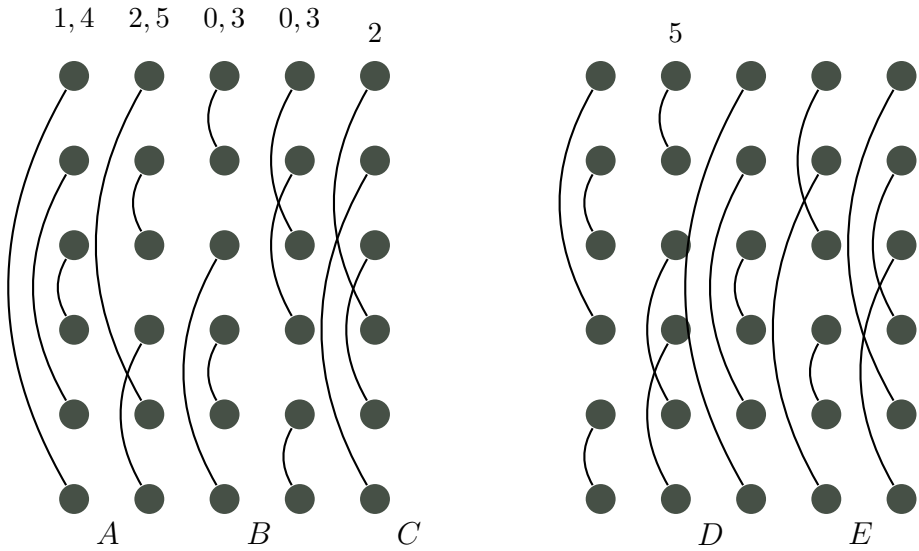
\begin{figure}[h!]
\centering
\definecolor{mygray}{RGB}{70,80,70}

\begin{tikzpicture}[thick,
  every node/.style={circle},
  fsnode/.style={fill=mygray},
  gsnode/.style={fill=mygray},
  hsnode/.style={fill=mygray},
  isnode/.style={fill=mygray}
]

\begin{scope}[start chain=going below,node distance=7mm]
  \node[fsnode,on chain] (f1) [label={above:\small ${1,4}$}] {};
  \node[fsnode,on chain] (f2) {};
  \node[fsnode,on chain] (f3) {};
  \node[fsnode,on chain] (f4) {};
  \node[fsnode,on chain] (f5) {};
  \node[fsnode,on chain] (f6) [label={below right:${A}$}]{};
\end{scope}

\begin{scope}[xshift=1cm,yshift=0cm,start chain=going below,node distance=7mm]
  \node[gsnode,on chain] (g1) [label={above:\small ${2,5}$}] {};
  \node[gsnode,on chain] (g2) {};
  \node[gsnode,on chain] (g3) {};
  \node[gsnode,on chain] (g4) {};
  \node[gsnode,on chain] (g5) {};
  \node[gsnode,on chain] (g6) {};
\end{scope}

\begin{scope}[xshift=2cm,yshift=0cm,start chain=going below,node distance=7mm]
  \node[hsnode,on chain] (h1) [label={above:\small ${0,3}$}]{};
  \node[hsnode,on chain] (h2) {};
  \node[hsnode,on chain] (h3) {};
  \node[hsnode,on chain] (h4) {};
  \node[hsnode,on chain] (h5) {};
  \node[hsnode,on chain] (h6) [label={below right:${B}$}]{};
\end{scope}

\begin{scope}[xshift=3cm,yshift=0cm,start chain=going below,node distance=7mm]
  \node[isnode,on chain] (i1) [label={above:\small ${0,3}$}] {};
  \node[isnode,on chain] (i2) {};
  \node[isnode,on chain] (i3) {};
  \node[isnode,on chain] (i4) {};
  \node[isnode,on chain] (i5) {};
  \node[isnode,on chain] (i6) {};
\end{scope}

\begin{scope}[xshift=4cm,yshift=0cm,start chain=going below,node distance=7mm]
  \node[hsnode,on chain] (j1) [label={above:\small ${2}$}]{};
  \node[hsnode,on chain] (j2) {};
  \node[hsnode,on chain] (j3) {};
  \node[hsnode,on chain] (j4) {};
  \node[hsnode,on chain] (j5) {};
  \node[hsnode,on chain] (j6) [label={below right:${C}$}] {};
\end{scope}

\begin{scope}[xshift=7cm,yshift=0cm,start chain=going below,node distance=7mm]
  \node[isnode,on chain] (k1) {};
  \node[isnode,on chain] (k2) {};
  \node[isnode,on chain] (k3) {};
  \node[isnode,on chain] (k4) {};
  \node[isnode,on chain] (k5) {};
  \node[isnode,on chain] (k6) {};
\end{scope}

\begin{scope}[xshift=8cm,yshift=0cm,start chain=going below,node distance=7mm]
  \node[hsnode,on chain] (l1) [label={above:\small ${5}$}]{};
  \node[hsnode,on chain] (l2) {};
  \node[hsnode,on chain] (l3) {};
  \node[hsnode,on chain] (l4) {};
  \node[hsnode,on chain] (l5) {};
  \node[hsnode,on chain] (l6) [label={below right:${D}$}] {};
\end{scope}

\begin{scope}[xshift=9cm,yshift=0cm,start chain=going below,node distance=7mm]
  \node[isnode,on chain] (m1) {};
  \node[isnode,on chain] (m2) {};
  \node[isnode,on chain] (m3) {};
  \node[isnode,on chain] (m4) {};
  \node[isnode,on chain] (m5) {};
  \node[isnode,on chain] (m6) {};
\end{scope}

\begin{scope}[xshift=10cm,yshift=0cm,start chain=going below,node distance=7mm]
  \node[hsnode,on chain] (n1) {};
  \node[hsnode,on chain] (n2) {};
  \node[hsnode,on chain] (n3) {};
  \node[hsnode,on chain] (n4) {};
  \node[hsnode,on chain] (n5) {};
  \node[hsnode,on chain] (n6) [label={below right:${E}$}] {};
\end{scope}

\begin{scope}[xshift=11cm,yshift=0cm,start chain=going below,node distance=7mm]
  \node[isnode,on chain] (o1) {};
  \node[isnode,on chain] (o2) {};
  \node[isnode,on chain] (o3) {};
  \node[isnode,on chain] (o4) {};
  \node[isnode,on chain] (o5) {};
  \node[isnode,on chain] (o6) {};
\end{scope}

\draw (f1) edge[bend right] node [left] {} node {} (f6);
\draw (f2) edge[bend right] node [left] {} node {} (f5);
\draw (f3) edge[bend right] node [left] {} node {}  (f4);
\draw (g1) edge[bend right] node [left] {} node {} (g5);
\draw (g2) edge[bend right] node [left] {} node {} (g3);
\draw (g4) edge[bend right] node [left] {} node {}  (g6);
\draw (h1) edge[bend right] node [left] {} node {} (h2);
\draw (h3) edge[bend right] node [left] {} node {} (h6);
\draw (h4) edge[bend right] node [left] {} node {}  (h5);
\draw (i1) edge[bend right] node [left] {} node {} (i3);
\draw (i2) edge[bend right] node [left] {} node {} (i4);
\draw (i5) edge[bend right] node [left] {} node {}  (i6);
\draw (j1) edge[bend right] node [left] {} node {} (j4);
\draw (j2) edge[bend right] node [left] {} node {} (j6);
\draw (j3) edge[bend right] node [left] {} node {}  (j5);

\draw (k1) edge[bend right] node [left] {} node {} (k4);
\draw (k2) edge[bend right] node [left] {} node {} (k3);
\draw (k5) edge[bend right] node [left] {} node {}  (k6);
\draw (l1) edge[bend right] node [left] {} node {} (l2);
\draw (l3) edge[bend right] node [left] {} node {} (l5);
\draw (l4) edge[bend right] node [left] {} node {}  (l6);
\draw (m1) edge[bend right] node [left] {} node {} (m6);
\draw (m2) edge[bend right] node [left] {} node {} (m5);
\draw (m3) edge[bend right] node [left] {} node {}  (m4);
\draw (n1) edge[bend right] node [left] {} node {} (n3);
\draw (n2) edge[bend right] node [left] {} node {} (n6);
\draw (n4) edge[bend right] node [left] {} node {}  (n5);
\draw (o1) edge[bend right] node [left] {} node {} (o5);
\draw (o2) edge[bend right] node [left] {} node {} (o4);
\draw (o3) edge[bend right] node [left] {} node {}  (o6);
\end{tikzpicture}

\caption{A 1-factorization of $K_6$ on $G_0$, and a 1-factorization of $K_6$ on $G_1$ }
\label{8}
\end{figure}

\

When $t \geq 2$, we provide a decomposition of $K_{(6:2t)} \cup 2tK_6$ into a 6-cycle factor, $(6t-2)$ 4-cycle factors, and a 1-factor. Let $H'=K_{2t}$ and let $F_1,F_2, \ldots, F_{2t-1}$ a 1-factorization of $H'$. Without loss, we may assume that $F_1= \{0,1\}, \{2,3\}, \ldots, \{2t-2,2t-1\}$. Then on each edge of $F_1$, create a 2-cycle, $H_1=(0,1), H_2=(2,3) \ldots, H_{2t-1}=(2t-2,2t-1)$. Apply case 1 to each of these cycles to obtain a set of four 4-cycle factors, one 6-cycle factor, and a 1-factor.

Now consider $F_i$ for $i= 2,3, \ldots, 2t-1$. On each of $F_i$, create a 2-cycle $H$. On each 2-cycle, let $d_1=1$ and $d_2=2$. Then $d \equiv 3 \pmod{6}$, so by Lemma \ref{Lemma2} $G[\Delta]$ produces a 4-cycle factor on $\bigcup_{i=0}^{2t-1} G_i$.

On each 2-cycle, let $d_1=0$ and $d_2=3$ to obtain another 4-cycle factor. Finally, on each 2-cycle, let $d_1=2$ and $d_2=1$ to obtain a third 4-cycle factor. Then on each $F_i$, we have covered each cross difference from groups $G_k$ to $G_{k+1}$ for $k=0,2,4, \ldots,2t-2$.

Furthermore, because we took a 1-factorization of $H'$, we have also covered cross differences between every pair of groups. This produces a total of $4+3(2t-2)=(6t-2)$ 4-cycle factors, one 6-cycle factor, and a 1-factor.

\end{proof}

\

The following result is obtained by using a variant of Wilson's Fundamental construction \cite{WI}.

For a cycle of length $a$, $C_a$, we give the subgraph, $C_{(w:a)}$ of $K_{(w:a)}$ obtained by giving each vertex of $C_a$ a {\it weight}. Let $w$ be a positive integer called the weight, and let $W=\{1,2, \ldots,w \}$. Let $V(C_{(w:a)})=\{\{x\} \times W | x \in V(C_a)\}$ and $E(C_{(w:a)})=\{\{(x_1,i), (x_2,j)\} | \{x_1,x_2\} \in E(C_a) \ \hbox{and} \\ \ i,j \in W\}$.

\begin{lemma}
\label{16a}
Let $G$ be a $C_a$-factorization of $K_{(x:t)}$. Suppose that for each cycle $C_a \in G$, there exists a $C_b$-factorization of $C_{(w:a)}$. Then the graph $K_{(wx:t)}$ has a $C_b$-factorization.
\end{lemma}

\begin{proof}
    Since $K_{(x:t)}$ has $xt$ vertices and $x^2 \binom{t}{2}$ edges, and each $C_a$-factor uses $xt$ edges, it follows that there are $\frac{x^2 \binom{t}{2}}{xt}= \frac{x(t-1)}{2}$ factors in G. For each cycle $C_a \in F_i$, $i=1,2, \ldots, \frac{x(t-1)}{2}$, define a $C_b$-factorization of $C_{(w:a)}$. Consider any edge $\{(x_1,i),(x_2,j)\} \in K_{(wx:t)}$. The edge $\{x_1,x_2\}$ is contained in exactly one cycle, $C_a \in G$, since $G$ is a $C_a$-factorization of $K_{(x:t)}$. Then, the edge $\{(x_1,i),(x_2,j)\}$ is contained in exactly one cycle, $C_b$, in the $C_b$-factorization of $C_{(w:a)}$. Each $C_b$-factorization of $C_{(w:a)}$ has $w$ factors, so each $F_i$ gives rise to $w$ $C_b$-factors. Thus, our decomposition produces $\frac{wx(t-1)}{2}$ $C_b$-factors.
\end{proof}

\begin{lemma}
\label{16b}
    If $x \equiv 2 \ \hbox{or} \ 4 \pmod{6}$, there exists a $C_{2x}$-factorization of $C_{(6:x)}$
\end{lemma}

\begin{proof}

For the cycle $C_x=(i_1, i_2, \ldots, i_x)$, we give a $C_{2x}$-factorization of $C_{(6:x)}$. Suppose there exists a $MRSM_{\mathbb Z_6} (S,6,x; \Sigma)$ with $S= \mathbb{Z}_6$. Then for each row of the matrix $(d_{1} d_{2} \ldots d_{x})$, build the graph $G[\Delta]$, where $\Delta = (D_{1}, D_{2}, \ldots, D_{x})$. If $d= \sum_{i=1}^{x} d_{i} \equiv 3 \pmod{6}$, then by Lemma \ref{Lemma2}, $G[\Delta]$ consists of a $2x$-cycle factor. Because each column of the matrix is a permutation of $S$, it follows that between any pair of groups $G_u$ and $G_v$, for $(u,v) \in C_x$, every edge with cross difference $d_u \in S$ appears exactly once in a $2x$-cycle. Thus we proceed by building the desired matrices. If $x=2$, use matrix $A$; given in Table \ref{table A}, and if $x=4$, use matrix $B$; given in Table \ref{table B}.

\begin{table}
\centering
\begin{tabularx}{0.8\textwidth} { 
  | >{\centering\arraybackslash}X 
  | >{\centering\arraybackslash}X 
  | >{\centering\arraybackslash}X | }
 \hline
 $(j_1,j_2)$ & $(j_2,j_1)$ \\
 \hline
 0 & 3 \\
 3 & 0 \\
 2 & 1 \\
 1 & 2 \\
 4 & 5 \\
 5 & 4 \\
\hline
\end{tabularx}
\caption{ A}
    \label{table A}
\end{table}

\begin{table}
\centering
\begin{tabularx}{0.8\textwidth} { 
  | >{\centering\arraybackslash}X 
  | >{\centering\arraybackslash}X 
  | >{\centering\arraybackslash}X
  | >{\centering\arraybackslash}X | }
 \hline
 $(j_1,j_2)$ & $(j_2,j_3)$ & $(j_3,j_4)$ & $(j_4,j_1)$ \\
 \hline
 0 & 0 & 0 & 3 \\
 3 & 3 & 3 & 0 \\
 2 & 2 & 4 & 1 \\
 4 & 1 & 2 & 2 \\
 5 & 4 & 1 & 5 \\
 1 & 5 & 5 & 4 \\
\hline
\end{tabularx}
\caption{B}
    \label{table B}
\end{table}
\

    If $x=6t+2$ and $t \geq 1$, then we build a $MRSM_{\mathbb Z_6} (S,6,x; \Sigma)$ $T$ by concatenating $t$ copies of matrix $B$ and $t+1$ copies of matrix $A$. Then $T$ has $4t+2(t+1)=6t+2$ columns. Furthermore, the number of copies of matrices used is $t+t+1=2t+1$, which is odd. Therefore, the sum of each row of $T$ is $3(2t+1) \equiv 3 \pmod{6}$. 

\
    If $x=6t+4$ and $t \geq 1$, then we build a $MRSM_{\mathbb Z_6} (S,6,x; \Sigma)$ $T'$ by concatenating $t$ copies of matrix  $A$ and $t+1$ copies of matrix $B$. Then $T'$ has $2t+4(t+1)=6t+4$ columns. As in the previous case, the number of copies of matrices used is $2t+1$, so the sum of each row of $T'$ is $3(2t+1) \equiv 3 \pmod{6}$. 
\end{proof}

\
\begin{theorem}
\label{1-gcd=2}
 There exists a solution to ${\hbox{HWP}}(6xt;6,2x;1,3xt-2)$, where $x\equiv 2 \ \hbox{or} \ 4 \ \hbox{(mod 6)}$.
\end{theorem}

\begin{proof}
If $x=2$ the result was proved in Theorem \ref{x=2}. 

Suppose $x>2$ and $t=1$. We will prove that there exists a solution to  $\hbox{HWP}(6x;6,2x;1,3x-2)$. 

\

Consider the equipartite graph $K_{(6:x)}$. As described in the discussion following Lemma \ref{Lemma2}, each row of a given $MRSM_{\mathbb{Z}_6} (S,3,x;\Sigma)$ $T$ corresponds to a subgraph, $G[\Delta]$ of $K_{(6:x)}$. In particular, the row $(d_{i_1},d_{i_2}, \ldots, d_{i_x})$ gives rise to $G[\Delta]$, where $\Delta=(D_{i_1},D_{i_2}, \ldots, D_{i_x})$. In order to find a decomposition of $K_{(6:x)}$, we must ensure that we cover all edges of each cross difference between every pair of distinct parts $G_u$ and $G_v$ with $u,v \in \{0,1,\ldots, x-1\}$. Then to find a solution to $\hbox{HWP}(6x;6,2x;1,3x-2)$, we must also cover all inside edges $G_{i_k}[K_6]$ for $k=1,2, \ldots, x$. 

\

If $x \equiv 4,8 \pmod{12}$, then by Lemma \ref{Tables1a} there exists a set of $x-1$ $MRSM_{\mathbb{Z}_6} (S,3,x;\Sigma)$ such that the entries in each row of the $x-2$ matrices, $T_0,T_1, \ldots, T_{x-3}$ sum to $d \equiv 3 \pmod{6}$. Therefore, by Lemma \ref{Lemma2} each of these rows gives rise to a subgraph, $G[\Delta]$, which is a $2x$-cycle factor. The matrix $T_{x-2}$ contains two rows whose sum is $d \equiv 3 \pmod{6}$, which contribute two more $2x$-cycle factors. The remaining row of $T_{x-2}$ has sum $d \pmod{6}$, with $d$ even. Thus by Lemma \ref{mix}, $G[\Delta] {\bigcup_{k=1}^{x}} G_{i_k}[f_1 \cup f_2]$ has a decomposition into two more $2x$-cycle factors, which produces a total of $3(x-2)+2+2=(3x-2)$ $2x$-cycle factors.

\

If $x \equiv 2,10 \pmod{12}$, then by Lemma \ref{Tables1b} there exists a set of $x-1$ $MRSM_{\mathbb{Z}_6} (S,3,x;\Sigma)$ such that the entries in each row of the following $x-3$ matrices, $T_1,T_2, \ldots, T_{x-3}$ sum to $d \equiv 3 \pmod{6}$. Therefore, by Lemma \ref{Lemma2} each of these rows gives rise to a subgraph, $G[\Delta]$, which is a $2x$-cycle factor. The matrix $T_0$ contains two rows whose sum is $d \equiv 3 \pmod{6}$, which contribute two more $2x$-cycle factors. Matrix $T_0$ also contains a row whose sum is 0. Thus by Lemma \ref{row0s}, $G[\Delta] {\bigcup_{k=1}^{x}} G_{i_k}[f_5]$ has a decomposition into one more $2x$-cycle factor. The matrix $T_{x-2}$ contains two rows whose sum is $d \equiv 3 \pmod{6}$, which contribute two more $2x$-cycle factors; and it also contains a row whose sum is $d \pmod{6}$, $d$ even. Thus by Lemma \ref{mix}, $G[\Delta] {\bigcup_{k=1}^{x}} G_{i_k}[f_1 \cup f_2]$ has a decomposition into two more $2x$-cycle factors, which produces a total of $3(x-3)+2+1+2+2=(3x-2)$ $2x$-cycle factors. 

\

Recall that the columns of each matrix $T_j$, $j = 0, 1, \ldots, x-2$, given in Lemmas \ref{Tables1a} and \ref{Tables1b} are induced by the $x$ edges in the Hamilton cycle $C_j$ that $T_j$ corresponds to. Furthermore, the $x-1$ cycles $C_j$, $j= 0,1,\ldots, x-2$ decompose $2K_x$. Therefore, among the set of matrices $(T_0,T_1,\ldots, T_{x-1})$, we see each edge $\{u,v\} \in K_x$ exactly twice as a column index. However, the union of these two columns contains every cross difference $d_u \in \{0,1,2,3,4,5\}$ exactly once. Thus we have covered all edges between every pair of distinct parts $G_u$ and $G_v$ exactly once. Furthermore, we have also used the edges $G_{i_k}[f_1 \cup f_2]$ for $k=1,2, \ldots, x$. The remaining edges are exactly $G_{i_k}[f_3 \cup f_4 \cup f_5]$ for $k=1,2, \ldots, x$, which are illustrated in Figure \ref{1-factors}. It is easy to see that these edges can be decomposed into a 6-cycle factor $(f_3 \cup f_4)$ and a 1-factor $(f_5)$. Thus we have a solution to $\hbox{HWP}(6x;6,2x;1,3x-2)$.

\

 Now suppose $t\geq2$, we will prove that there exists a solution to $\hbox{HWP}(6xt;6,2x;1,3xt-2)$. 

\

Consider the equipartite graph $K_{(x:t)}$, which has $t$ parts of size $x$, and and let $x\equiv 2 \ \hbox{or} \ 4 \ \pmod{6}$. By Theorem \ref{equip}, we know that $K_{(x:t)}$ has a resolvable $C_x$-factorization. We know by Lemma \ref{16b} that there exists a $C_{2x}$-factorization of $C_{(6:x)}$. So by Lemma \ref{16a}, there exists a $C_{2x}$-factorization of $K_{(6x:t)}$. This has $3x(t-1)$ $C_{2x}$-factors. On each part of $K_{(6x:t)}$, use the solution to $\hbox{HWP}(6x;6,2x;1,3x-2)$. This produces a total of $3xt-2$ $C_{2x}$-factors, one $C_6$-factor, and a 1-factor.
\end{proof}

\subsection{$\gcd(6,N)=2$ and $x \equiv 1 \ \hbox{or} \ 5 \pmod{6}$}

Suppose $X$ is a graph on $n$ vertices. Let $\{a,b\}$ and $\{c,d\}$ be a pair of independent edges of $X$. Let $C$ be a 4-cycle on the vertices $\{a,b,c,d\}$ which alternates between edges in $X$ and edges not in $X$. Then, we will refer to $E(X) \oplus E(C)$ as a $4$-cycle switch of $X$ with $C$. Similarly, let $\{a,b\}$, $\{c,d\}$, and $\{e,f\}$ be a set of three independent edges of $X$. Let $C$ be a 6-cycle on the vertices $\{a,b,c,d,e,f\}$ which alternates between edges in $X$ and edges not in $X$. Then, we will refer to $E(X) \oplus E(C)$ as a $6$-cycle switch of $X$ with $C$.

\begin{lemma}
\label{Lemma A}
    If $X$ is a cycle of length $n$, and $(a,b)$, $(c,d)$, and $(e,f)$ are three independent edges in $X$, then a 6-cycle switch of $X$ with $C=(a,b,e,f,c,d)$ produces another cycle of length $n$.
\end{lemma}

\begin{proof}
    The cycle $X$ can be described by the three independent edges and the three paths that separate them. $$X= (a,b), P_1, (c,d), P_2, (e,f), P_3$$

where $P_1$ is a $b-c$ path, $P_2$ is a $d-e$ path, and $P_3$ is an $f-a$ path.

Then, $E(X) \oplus E(C)$  produces the following cycle of length $n$: $(a,d), P_2, (e,b), P_1, (c,f), P_3$.
\end{proof}

\begin{lemma}
\label{Lemma B}
    If $X$ is a cycle of length $n$, and $(a,b)$, $(c,d)$, and $(e,f)$ are three independent edges in $X$, then a 6-cycle switch of $X$ with $C=(a,b,c,d,e,f)$ produces a set of three independent spanning cycles.
\end{lemma}

\begin{proof}
    The cycle $X$ can be described by the three independent edges and the three paths that separate them. $$X= (a,b), P_1, (c,d), P_2, (e,f), P_3$$

where $P_1$ is a $b-c$ path, $P_2$ is a $d-e$ path, and $P_3$ is a $f-a$ path.

Then, $E(X) \oplus E(C)$  produces the following three independent spanning cycles:  
$$(f,a), \Bar{P_3}$$
$$(b,c), \Bar{P_1}$$
$$(d,e), \Bar{P_2},$$
\end{proof}

where $\Bar{P_i}$ denotes the path in $X$ following the opposite direction of $P_i$.

\begin{lemma}
\label{Lemma C}
    If $X$ is a cycle of length $n$, and $(a,b)$ and $(c,d)$ are two independent edges in $X$, then a 4-cycle switch of $X$ with $C=(b,a,c,d)$ produces another cycle of length $n$.
\end{lemma}

\begin{proof}
    The cycle $X$ can be described by the two independent edges and the two paths that separate them. $$X= (a,b), P_1, (c,d), P_2$$

where $P_1$ is a $b-c$ path, $P_2$ is a $d-a$ path.

Then, $E(X) \oplus E(C)$  produces the following cycle of length $n$: $(a,c), \Bar{P_1}, (b,d), P_2$, where $\Bar{P_1}$ denotes the path in $X$ following the opposite direction of $P_1$.
\end{proof}

\

In \cite{WE}, the following result was given as a Remark.

\begin{lemma}
\label{Remark 4.4}
Suppose $C_1$ and $C_2$ are vertex disjoint cycles in a graph $X$, and $\{x_i, y_i\} \in E(C_i)$ for $i \in \{1,2\}$. If $C=x_1 y_1 x_2 y_2$ is a cycle of length four in $X$, and the edges $\{y_1, x_2\}$ and $\{y_2, x_1\}$ are not $E(C_1) \cup E(C_2)$, then the subgraph of $X$ whose edge set is the symmetric difference $(E(C_1) \cup E(C_2)) \oplus E(C)$ is a single cycle.     
\end{lemma}

Similarly, we can extend this to the case of three vertex disjoint cycles.

\begin{lemma}
\label{Lemma D}
Suppose $C_1$, $C_2$, and $C_3$ are vertex disjoint cycles in a graph $X$, and $\{x_i, y_i\} \in E(C_i)$ for $i \in \{1,2,3\}$. If $C=x_1 y_1 x_2 y_2 x_3 y_3$ is a cycle of length six in $X$, and the edges $\{y_1, x_2\}$, $\{y_2, x_3\}$, and $\{y_3, x_1\}$ are not $E(C_1) \cup E(C_2) \cup E(C_3)$, then the subgraph of $X$ whose edge set is the symmetric difference $(E(C_1) \cup E(C_2) \cup E(C_3) \oplus E(C)$ is a single cycle.     
\end{lemma}
\

It is well known that a decomposition of $K_x$ into Hamilton cycles exists for all odd $x$. The result was given by Lucas in \cite{LU}, but it was attributed to Walecki. The construction is described in detail in \cite{AB}. We give the construction as follows.

\

Label the vertices of $K_x$ as $\{u_0,u_1, \hdots, u_{x-1}\}$, and let $\sigma$ be the permutation whose disjoint cycle decomposition is
$$\sigma=(u_0)(u_1,u_2,\hdots,u_{x-1}).$$

Let $H_1$ be the Hamilton cycle
$$H_1=(u_0,u_1,u_2,u_{x-1},u_3,u_{x-2},\hdots, u_{\frac{x-1}{2}},u_{\frac{x-1}{2}+2}, u_{\frac{x-1}{2}+1}), \ \hbox{and}$$

let $H_i={\sigma}^{i-1}(H_1)$, $i=1,2,\hdots,\frac{x-1}{2}$. Then our decomposition of $K_x$ is given by

$$H_1\bigoplus H_2\bigoplus, \cdots, \bigoplus H_{\frac{x-1}{2}}.$$

\

An example of $H_1$ for $x=9$ is given in Figure \ref{x=9}. It is easy to see that by rotating the cycle clockwise 3 times gives the desired decomposition. This is described by the permutations $\sigma$, $\sigma^2$, and $\sigma^3$.

\

\begin{figure}[h!]
\centering  
\usetikzlibrary {positioning}
\definecolor{mygray}{RGB}{70,80,70}
\begin{tikzpicture}[ 
every node/.style={fill=mygray, circle, inner sep=0.05cm}
]
  \begin{scope}[node distance=15mm and 15mm]
    \node (b) at (2,1) [label={above right:$u_0$}] {};
    \node (a) at (0,1) [label={above right:$u_7$}] {};       
    \node (c) at (4,1) [label={above right:$u_3$}] {};
    \node (d) at (2,3) [label={above right:$u_1$}] {};      
    \node (e) at (2,-1) [label={above right:$u_5$}] {};
    \node (f) at (0.5,2.4) [label={above right:$u_8$}] {}; 
    \node (g) at (3.4,2.4) [label={above right:$u_2$}] {};
    \node (h) at (0.5, -0.5) [label={above right:$u_6$}] {}; 
    \node (i) at (3.4,-0.5) [label={above right:$u_4$}] {};
  \end{scope}

\draw (b) -- (d);
\draw (d) -- (g);
\draw (g) -- (f);
\draw (f) -- (c);
\draw (c) edge [bend right] (a);
\draw (a) -- (i);
\draw (i) -- (h);
\draw (h) -- (e);
\draw (e) -- (b);

\end{tikzpicture}

\caption{A Hamilton cycle, $H_1$, of $K_9$.}
\label{x=9}
\end{figure}
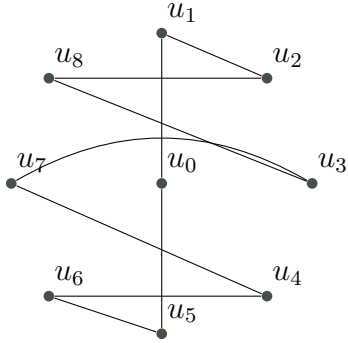

\

We will use the Walecki decomposition along with cycle switches to obtain the next result.

\begin{theorem}
\label{2-gcd=2}
 There exists a solution to ${\hbox{HWP}}(6x;6,2x;1,3x-2)$, where $x\equiv 1 \ \hbox{or} \ 5 \ \hbox{(mod 6)}$.
\end{theorem}

\begin{proof}

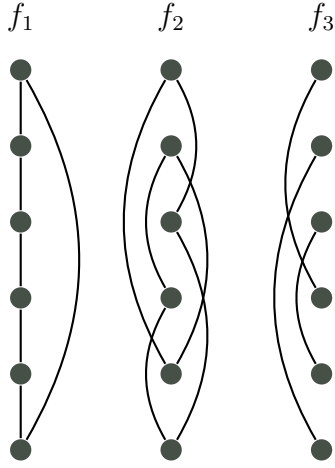
\begin{figure}[h!]
\centering
\definecolor{mygray}{RGB}{70,80,70}

\begin{tikzpicture}[thick,
  every node/.style={circle, inner sep=0.1cm},
  fsnode/.style={fill=mygray},
  gsnode/.style={fill=mygray},
  hsnode/.style={fill=mygray},
  isnode/.style={fill=mygray},
  jsnode/.style={fill=mygray}
]

\begin{scope}[start chain=going below,node distance=7mm]
\foreach \i in {1,2,...,6}
  \node[fsnode,on chain] (f\i) {};
\end{scope}

\begin{scope}[xshift=2cm,yshift=0cm,start chain=going below,node distance=7mm]
\foreach \i in {1,2,...,6}
  \node[gsnode,on chain] (g\i) {};
\end{scope}

\begin{scope}[xshift=4cm,yshift=0cm,start chain=going below,node distance=7mm]
\foreach \i in {1,2,...,6}
  \node[hsnode,on chain] (h\i) {};
\end{scope}

\node [fit=(f1) (f6),label=above:$f_1$] {};
\node [fit=(g1) (g6),label=above:$f_2$] {};
\node [fit=(h1) (h6),label=above:$f_3$] {};

\draw (f1) -- (f2);
\draw (f2) -- (f3);
\draw (f3) -- (f4);
\draw (f4) -- (f5);
\draw (f5) -- (f6);
\draw (f6) edge[bend right] node [left] {}(f1);
\draw (g1) edge[bend right] node [left] {}(g5);
\draw (g5) edge[bend right] node [left] {}(g2);
\draw (g2) edge[bend right] node [left] {}(g4);
\draw (g4) edge[bend right] node [left] {}(g6);
\draw (g6) edge[bend right] node [left] {}(g3);
\draw (g3) edge[bend right] node [left] {}(g1);
\draw (h1) edge[bend right] node [left] {}(h4);
\draw (h2) edge[bend right] node [left] {}(h6);
\draw (h3) edge[bend right] node [left] {}(h5);
\end{tikzpicture}
\caption{$K_6$ into two 6-cycles and a 1-factor.}
\label{K6}
\end{figure}

 Let $G_{0},G_{1},G_{2},G_{3},G_{4},G_{5}$ be the six parts of $K_{(x:6)}$ with $G_i$ having vertex set $\{(i,j): 0\leq j \leq x-1\}$. We will consider the decomposition of $K_{(x:6)} \cup 6K_x$. 

Let $F= \cup_{j=1}^5 F_j$, $F^*= \cup_{j=1}^5 F^*_j$, and $F'= \cup_{j=1}^5 F_j'$ be three different 1-factorizations of $K_6$, which are given in Figure \ref{9}. Define $c=(i_1,i_2)$ to be a 2-cycle corresponding to some edge $\{i_1,i_2\} \in F_1$. Then for any $c=(i_1,i_2)$, let $G_c = K_{(x:2)}$. Because $F_1=\{\{(0,1\}, \{2,3\}, \{4,5\}\}$, we consider the graph $G= G_{c_1} \cup G_{c_2} \cup G_{c_3}$ where $c_1=(0,1)$, $c_2=(2,3)$, and $c_3=(4,5)$.

Let $H = H_1\bigoplus H_2\bigoplus, \cdots, \bigoplus H_{\frac{x-1}{2}}$ be the Walecki decomposition of $K_x$, and let $k=1,2,\hdots,\frac{x-1}{2}$. Now, consider the following two cases, $k$ even and $k$ odd.

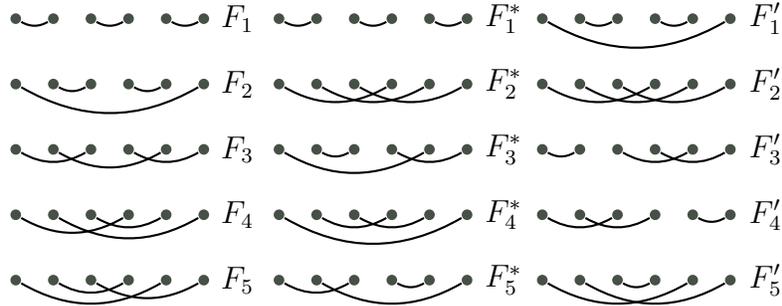
\begin{figure}[h!]
\centering
\definecolor{mygray}{RGB}{70,80,70}

\begin{tikzpicture}[thick,
  every node/.style={circle, inner sep=0.05cm},
  fsnode/.style={fill=mygray},
  gsnode/.style={fill=mygray},
  hsnode/.style={fill=mygray},
  isnode/.style={fill=mygray}
]

\begin{scope}[start chain=going below,node distance=7mm]
  \node[fsnode,on chain] (f1) {};
  \node[fsnode,on chain] (f2) {};
  \node[fsnode,on chain] (f3) {};
  \node[fsnode,on chain] (f4) {};
  \node[fsnode,on chain] (f5) {};
\end{scope}

\begin{scope}[xshift=.5cm,yshift=0cm,start chain=going below,node distance=7mm]
  \node[gsnode,on chain] (g1) {};
  \node[gsnode,on chain] (g2) {};
  \node[gsnode,on chain] (g3) {};
  \node[gsnode,on chain] (g4) {};
  \node[gsnode,on chain] (g5) {};
\end{scope}

\begin{scope}[xshift=1cm,yshift=0cm,start chain=going below,node distance=7mm]
  \node[hsnode,on chain] (h1) {};
  \node[hsnode,on chain] (h2) {};
  \node[hsnode,on chain] (h3) {};
  \node[hsnode,on chain] (h4) {};
  \node[hsnode,on chain] (h5) {};
\end{scope}

\begin{scope}[xshift=1.5cm,yshift=0cm,start chain=going below,node distance=7mm]
  \node[isnode,on chain] (i1) {};
  \node[isnode,on chain] (i2) {};
  \node[isnode,on chain] (i3) {};
  \node[isnode,on chain] (i4) {};
  \node[isnode,on chain] (i5) {};
\end{scope}

\begin{scope}[xshift=2cm,yshift=0cm,start chain=going below,node distance=7mm]
  \node[hsnode,on chain] (j1) {};
  \node[hsnode,on chain] (j2) {};
  \node[hsnode,on chain] (j3) {};
  \node[hsnode,on chain] (j4) {};
  \node[hsnode,on chain] (j5) {};
\end{scope}

\begin{scope}[xshift=2.5cm,yshift=0cm,start chain=going below,node distance=7mm]
  \node[isnode,on chain] (k1) [label={right:${F_1}$}] {};
  \node[isnode,on chain] (k2) [label={right:${F_2}$}] {};
  \node[isnode,on chain] (k3) [label={right:${F_3}$}] {};
  \node[isnode,on chain] (k4) [label={right:${F_4}$}] {};
  \node[isnode,on chain] (k5) [label={right:${F_5}$}] {};
\end{scope}

\begin{scope}[xshift=3.5cm,yshift=0cm,start chain=going below,node distance=7mm]
  \node[hsnode,on chain] (l1) {};
  \node[hsnode,on chain] (l2) {};
  \node[hsnode,on chain] (l3) {};
  \node[hsnode,on chain] (l4) {};
  \node[hsnode,on chain] (l5) {};
\end{scope}

\begin{scope}[xshift=4cm,yshift=0cm,start chain=going below,node distance=7mm]
  \node[isnode,on chain] (m1) {};
  \node[isnode,on chain] (m2) {};
  \node[isnode,on chain] (m3) {};
  \node[isnode,on chain] (m4) {};
  \node[isnode,on chain] (m5) {};
\end{scope}

\begin{scope}[xshift=4.5cm,yshift=0cm,start chain=going below,node distance=7mm]
  \node[hsnode,on chain] (n1) {};
  \node[hsnode,on chain] (n2) {};
  \node[hsnode,on chain] (n3) {};
  \node[hsnode,on chain] (n4) {};
  \node[hsnode,on chain] (n5) {};
\end{scope}

\begin{scope}[xshift=5cm,yshift=0cm,start chain=going below,node distance=7mm]
  \node[isnode,on chain] (o1) {};
  \node[isnode,on chain] (o2) {};
  \node[isnode,on chain] (o3) {};
  \node[isnode,on chain] (o4) {};
  \node[isnode,on chain] (o5) {};
\end{scope}

\begin{scope}[xshift=5.5cm,yshift=0cm,start chain=going below,node distance=7mm]
  \node[isnode,on chain] (p1) {};
  \node[isnode,on chain] (p2) {};
  \node[isnode,on chain] (p3) {};
  \node[isnode,on chain] (p4) {};
  \node[isnode,on chain] (p5) {};
\end{scope}

\begin{scope}[xshift=6cm,yshift=0cm,start chain=going below,node distance=7mm]
  \node[isnode,on chain] (q1) [label={right:${F_1^*}$}] {};
  \node[isnode,on chain] (q2) [label={right:${F_2^*}$}] {};
  \node[isnode,on chain] (q3) [label={right:${F_3^*}$}] {};
  \node[isnode,on chain] (q4) [label={right:${F_4^*}$}] {};
  \node[isnode,on chain] (q5) [label={right:${F_5^*}$}] {};
\end{scope}

\begin{scope}[xshift=7cm,yshift=0cm,start chain=going below,node distance=7mm]
  \node[isnode,on chain] (r1) {};
  \node[isnode,on chain] (r2) {};
  \node[isnode,on chain] (r3) {};
  \node[isnode,on chain] (r4) {};
  \node[isnode,on chain] (r5) {};
\end{scope}

\begin{scope}[xshift=7.5cm,yshift=0cm,start chain=going below,node distance=7mm]
  \node[isnode,on chain] (s1) {};
  \node[isnode,on chain] (s2) {};
  \node[isnode,on chain] (s3) {};
  \node[isnode,on chain] (s4) {};
  \node[isnode,on chain] (s5) {};
\end{scope}

\begin{scope}[xshift=8cm,yshift=0cm,start chain=going below,node distance=7mm]
  \node[isnode,on chain] (t1) {};
  \node[isnode,on chain] (t2) {};
  \node[isnode,on chain] (t3) {};
  \node[isnode,on chain] (t4) {};
  \node[isnode,on chain] (t5) {};
\end{scope}

\begin{scope}[xshift=8.5cm,yshift=0cm,start chain=going below,node distance=7mm]
  \node[isnode,on chain] (u1) {};
  \node[isnode,on chain] (u2) {};
  \node[isnode,on chain] (u3) {};
  \node[isnode,on chain] (u4) {};
  \node[isnode,on chain] (u5) {};
\end{scope}

\begin{scope}[xshift=9cm,yshift=0cm,start chain=going below,node distance=7mm]
  \node[isnode,on chain] (v1) {};
  \node[isnode,on chain] (v2) {};
  \node[isnode,on chain] (v3) {};
  \node[isnode,on chain] (v4) {};
  \node[isnode,on chain] (v5) {};
\end{scope}

\begin{scope}[xshift=9.5cm,yshift=0cm,start chain=going below,node distance=7mm]
  \node[isnode,on chain] (w1) [label={right:${F_1'}$}] {};
  \node[isnode,on chain] (w2) [label={right:${F_2'}$}] {};
  \node[isnode,on chain] (w3) [label={right:${F_3'}$}] {};
  \node[isnode,on chain] (w4) [label={right:${F_4'}$}] {};
  \node[isnode,on chain] (w5) [label={right:${F_5'}$}] {};
\end{scope}

\draw (f1) edge[bend right] node [left] {} node {} (g1);
\draw (h1) edge[bend right] node [left] {} node {}  (i1);
\draw (j1) edge[bend right] node [left] {} node {}  (k1);
\draw (f2) edge[bend right] node [left] {} node {} (k2);
\draw (g2) edge[bend right] node [left] {} node {}  (h2);
\draw (i2) edge[bend right] node [left] {} node {}  (j2);
\draw (f3) edge[bend right] node [left] {} node {} (h3);
\draw (g3) edge[bend right] node [left] {} node {}  (j3);
\draw (i3) edge[bend right] node [left] {} node {}  (k3);
\draw (f4) edge[bend right] node [left] {} node {} (i4);
\draw (g4) edge[bend right] node [left] {} node {}  (k4);
\draw (h4) edge[bend right] node [left] {} node {}  (j4);
\draw (f5) edge[bend right] node [left] {} node {} (j5);
\draw (g5) edge[bend right] node [left] {} node {}  (i5);
\draw (h5) edge[bend right] node [left] {} node {}  (k5);

\draw (l1) edge[bend right] node [left] {} node {} (m1);
\draw (n1) edge[bend right] node [left] {} node {} (o1);
\draw (p1) edge[bend right] node [left] {} node {} (q1);
\draw (l2) edge[bend right] node [left] {} node {} (o2);
\draw (m2) edge[bend right] node [left] {} node {} (p2);
\draw (n2) edge[bend right] node [left] {} node {} (q2);
\draw (l3) edge[bend right] node [left] {} node {} (p3);
\draw (m3) edge[bend right] node [left] {} node {} (n3);
\draw (o3) edge[bend right] node [left] {} node {} (q3);
\draw (l4) edge[bend right] node [left] {} node {} (q4);
\draw (m4) edge[bend right] node [left] {} node {} (o4);
\draw (n4) edge[bend right] node [left] {} node {} (p4);
\draw (l5) edge[bend right] node [left] {} node {} (n5);
\draw (m5) edge[bend right] node [left] {} node {} (q5);
\draw (o5) edge[bend right] node [left] {} node {} (p5);

\draw (r1) edge[bend right] node [left] {} node {} (w1);
\draw (s1) edge[bend right] node [left] {} node {} (t1);
\draw (u1) edge[bend right] node [left] {} node {} (v1);
\draw (r2) edge[bend right] node [left] {} node {} (u2);
\draw (s2) edge[bend right] node [left] {} node {} (v2);
\draw (t2) edge[bend right] node [left] {} node {} (w2);
\draw (r3) edge[bend right] node [left] {} node {} (s3);
\draw (t3) edge[bend right] node [left] {} node {} (v3);
\draw (u3) edge[bend right] node [left] {} node {} (w3);
\draw (r4) edge[bend right] node [left] {} node {} (t4);
\draw (s4) edge[bend right] node [left] {} node {} (u4);
\draw (v4) edge[bend right] node [left] {} node {} (w4);
\draw (r5) edge[bend right] node [left] {} node {} (v5);
\draw (s5) edge[bend right] node [left] {} node {} (w5);
\draw (t5) edge[bend right] node [left] {} node {} (u5);

\end{tikzpicture}

\caption{ Three 1-factorizations of $K_6$, $F$, $F^*$, and $F'$.}
\label{9}
\end{figure}

First suppose $k$ is odd. For $c=c_1,c_2,c_3$,  let $C_1=G_{i_1}[H_k]$ and $C_2=G_{i_2}[H_k]$ be two independent Hamilton cycles on $G_{i_1}$ and $G_{i_2}$. Then $\sigma^{k-1}((u_{i_1},{1}), (u_{i_1},{2})) \in E(C_1)$ and \\
$\sigma^{k-1}((u_{i_2},{\frac{x+3}{2}}), (u_{i_2},{\frac{x+1}{2}})) \in E(C_2)$. Thus,  $C= \sigma^{k-1}((u_{i_1},{1}), (u_{i_1},{2}),(u_{i_2},{\frac{x+3}{2}}), (u_{i_2},{\frac{x+1}{2}}))$ is a cycle of length four in $K_{(x:2)} \cup 2K_x$. However, 

\[e_1 = \sigma^{k-1} \{(u_{i_1},2),(u_{i_2},{\frac{x+3}{2}})\}\] and 

\[e_2 = \sigma^{k-1} \{(u_{i_2},{\frac{x+1}{2}}),(u_{i_1},{1})\}\] 
are not in $E(C_1) \cup E(C_2)$. Then by Lemma \ref{Remark 4.4}, the subgraph of $K_{(x:2)} \cup 2K_x$ whose edge set is $(E(C_1) \cup E(C_2)) \oplus E(C)$ is a single cycle of length $2x$. Thus this 4-cycle switch, $S_1$, joins two independent $x$-cycles into a cycle of length $2x$. Therefore, for each $k$ we obtain a $2x$-cycle factor, $\cal{A}$$_k$.

Now suppose $k$ is even. For $c=c_1,c_2,c_3$, let $C_1=G_{i_1}[H_k]$ and $C_2=G_{i_2}[H_k]$ be two independent Hamilton cycles on $G_{i_1}$ and $G_{i_2}$. Then $\sigma^{k-1}((u_{i_1},{\frac{x+1}{2}}), (u_{i_1},{\frac{x+3}{2}})) \in E(C_1)$ and $\sigma^{k-1}((u_{i_2},{2}), (u_{i_2},{1})) \in E(C_2)$. Thus,  $C= \sigma^{k-1}((u_{i_1}{\frac{x+1}{2}}), (u_{i_1},{\frac{x+3}{2}}), (u_{i_2},{2}), (u_{i_2},{1}))$ is a cycle of length four in $K_{(x:2)} \cup 2K_x$. However, 

\[e_3 = \sigma^{k-1} \{(u_{i_1},{\frac{x+3}{2}}), (u_{i_2},{2})\}\] and 

\[e_4 = \sigma^{k-1} \{(u_{i_2},{1}), (u_{i_1},{\frac{x+1}{2}})\}\] 
are not in $E(C_1) \cup E(C_2)$. So by Lemma \ref{Remark 4.4}, $(E(C_1) \cup E(C_2)) \oplus E(C)$ produces a single cycle of length $2x$.Thus this 4-cycle switch, $S_2$, joins two independent $x$-cycles into a cycle of length $2x$. Therefore, for each $k$ we obtain a $2x$-cycle factor, $\cal{B}$$_k$. 

\

For each $c$, let $d_{i_1}= \frac{x-1}{2}$ and $d_{i_2}= \frac{x-1}{2}$, then $d=x-1$. So by Lemma \ref{Lemma2}, $G_c[\Delta]$ consists of a $2x$-cycle, and $G[\Delta]$ is a $2x$-cycle factor. For $k$ odd, the independent edges $e_1= \sigma^{k-1} ((u_{i_1},1), (u_{i_2}, \frac{x+1}{2}))$ and $e_2= \sigma^{k-1} ((u_{i_1},2), (u_{i_2}, \frac{x+3}{2}))$, which are used in $A_k$, are also in $G_{c}[\Delta]$. Thus the 4-cycle switch, $S_{1}$, that was used to produce $A_{k}$, results in a 4-cycle switch of $G_c[\Delta]$ with $C= \sigma^{k-1} ((u_{i_2},\frac{x+1}{2}), (u_{i_1},1), (u_{i_1},2), (u_{i_2},\frac{x+3}{2}))$. Similarly, when $k$ is even, recall that $e_3= \sigma^{k-1} ((u_{i_1},\frac{x+3}{2}), (u_{i_2},2))$ and $e_4= \sigma^{k-1} ((u_{i_1},\frac{x+1}{2}), (u_{i_2},1))$, which are used in $B_{k}$, are also in $G_c[\Delta]$. However,
$\sigma^{k-1}((u_{i_1},{\frac{x+1}{2}}), (u_{i_1},{\frac{x+3}{2}}))$ and $\sigma^{k-1}((u_{i_2},{2}), (u_{i_2},{1}))$ are not in $G_{c}[\Delta]$. Thus the 4-cycle switch, $S_2$, induces the 4-cycle switch of $G_c[\Delta]$ with $C= \sigma^{k-1} ((u_{i_2},2), (u_{i_1},\frac{x+3}{2}), \\ (u_{i_1},\frac{x+1}{2}), (u_{i_2},1))$. In each case, by Lemma \ref{Lemma C}, this cycle switch results in a different $2x$-cycle. Therefore, after applying all of the cycle switches, the resulting graph is still a $2x$-cycle factor, which we will call $\cal{C}$.

\

Now consider the 6-cycle $(0,1,2,3,4,5)$ and the list of differences $(1,0,1,0,1,0)$, so $d=3$. By Lemma \ref{Lemma2}, $K_{(x:6)}[\Delta]$ is a Hamilton cycle, which we will call $X$. If $x \equiv 1 \pmod{6}$, then $(a,b)=((1,0),(2,0))$, $(c,d)=((3,0),(4,0))$, and $(e,f)=((5,0),(0,0))$ are three independents edges on $X$. By Lemma \ref{Lemma B}, the 6-cycle switch, $S_3$, of $X$ with $C=(a,b,c,d,e,f)$ produces a set of three independent spanning cycles. Furthermore, the paths $P_1, P_2, P_3$ all have the same length, and the three cycles are each of size $2x$. If $x \equiv 5 \pmod{6}$, then  $(a',b')=((1,x-1),(2,x-1))$, $(c',d')=((3,x-1),(4,x-1))$, and $(e',f')=((5,x-1),(0,x-1))$ are three independents edges on $X$. By Lemma \ref{Lemma A}, the 6-cycle switch, $S_4$, of $X$ with $C'=(a',b',e',f',c',d')$ produces a different Hamilton cycle, which we will call $Y$. Now apply the 6-cycle switch, $S_3$, of $Y$ with $C$. By Lemma \ref{Lemma B}, this switch results in a $2x$- cycle factor. Thus in each case the 6-cycle switches produce a $2x$-cycle factor, which we will call $\cal{D}$.

\

For each 2-cycle $c$ corresponding to an edge $i_1,i_2 \in F_1$, let $d_{i_1}=x-1$ and $d_{i_2}=0$, then $d=x-1$. So by Lemma \ref{Lemma2}, $G_c[\Delta]$ consists of a $2x$-cycle, and $G[\Delta]$ is a $2x$-cycle factor, which we will call $X^*$. The 6-cycle switch, $S_3$, induces the 6-cycle switch of $X^*$ with $C=(f,a,b,c,d,e)$. By Lemma \ref{Lemma D}, this results in a single cycle of length $6x$, which we will call $\cal{Z}$. If $x \equiv 5 \pmod{6}$, then the 6-cycle switch, $S_4$, uses edges $(f',a')=((0,x-1),(1,x-1))$, $(b',c')=((2,x-1),(3,x-1))$, and $(d',e')=((4,x-1),(5,x-1))$. Thus we must remove these edges from $\cal{Z}$. The cycle $\cal{Z}$ can be described as $\cal{Z}$ $ = (f',a'), P_1, (d',e'), P_2, (b',c'), P_3$. So we perform the 6-cycle switch of $\cal{Z}$ with $C=(f',a',d',e',b',c')$, $S_5$, which by Lemma \ref{Lemma B} produces a $2x$-cycle factor. If $x \equiv 1 \pmod{6}$, then we also apply $S_5$ to $\cal{Z}$ to obtain a $2x$-cycle factor. Thus in each case the 6-cycle switches produce a $2x$-cycle factor, which we will call $\cal{E}$.

\

Note that we have obtained $(\frac{x-1}{2})$ $2x$-cycle factors from $\cal{A}$$_k \cup \cal{B}$$_k$, alongside the $2x$-cycle factors $\cal{C}$, $\cal{D}$, and $\cal{E}$. Consequently, the total number of $2x$-cycle factors obtained stands at $(\frac{x+5}{2})$. Note that we have covered all inside edges from $K_x$ in the factors $\cal{A}$$_k$ and $\cal{B}$$_k$. The $2x$-cycle factor $\cal{C}$ covers all cross differences $\frac{x-1}{2}$ and $\frac{x+1}{2}$ from $G_{i_1}$ to $G_{i_2}$ for each edge $(i_1,i_2) \in F_1$. Similarly, the $2x$-cycle factor $\cal{D}$ covers cross difference one from $G_{i_1}$ to $G_{i_2}$ for each edge $(i_1,i_2) \in F_1$, while the factor $\cal{E}$ also covers cross difference $x-1$ from $G_{i_1}$ to $G_{i_2}$ for each edge $(i_1,i_2) \in F_1$. Additionally, factors $\cal{D}$ and $\cal{E}$ accounts for some cross differences zero, which will be addressed in subsequent paragraphs of this proof.

\

Our goal now is to cover the remaining cross differences from $G_{i_1}$ to $G_{i_2}$ for each edge $(i_1,i_2) \in F_1$. In particular this list is $L= \{2,3, \ldots, x-2\} \setminus \{\frac{x-1}{2}, \frac{x+1}{2}\}$. If $x \equiv 7, 11 \pmod{12}$, on each edge $(i_1,i_2) \in F_1$, let $d_1=2$ and $d_2=2$, so $d=4$. Then, by Lemma \ref{Lemma2}, $G[\Delta]$ consists of a $2x$-cycle factor. The cross difference $d_2$ from $G_{i_2}$ to $G_{i_1}$ is effectively the cross difference $x-2$ from $G_{i_1}$ to $G_{i_2}$. Thus we produce a $2x$-cycle factor that covers cross differences 2 and $x-2$ on each edge $(i_1,i_2) \in F_1$. Similarly, let $d_{1} = 3,5,7, \ldots, \frac{x-5}{2}, \frac{x+3}{2}, \frac{x+7}{2}, \ldots, x-4$, and $d_2=x-1-d_1$. In each case, $d=x-1$ so by Lemma \ref{Lemma2}, $G[\Delta]$ produces a $2x$-cycle factor. The cross difference $d_2$ from $G_{i_2}$ to $G_{i_1}$ is effectively the cross difference $d_1 +1$ from $G_{i_1}$ to $G_{i_2}$. Thus we have covered all cross differences in $L$. If $x \equiv 1, 5 \pmod{12}$, on each edge $(i_1,i_2) \in F_1$, let $d_{1} = 2,4,6, \ldots, \frac{x-5}{2}, \frac{x+3}{2}, \frac{x+7}{2}, \ldots,x-3$, and $d_2=x-1-d_1$. Again, in each case we obtain $2x$-cycle factors, and have covered all cross differences in $L$. Let $\cal{F}$$_l$ denote the $2x$-cycle factors so obtained from the $l= \frac{x-5}{2}$ items on $L$.

\

We now turn to $F_2$, $F_3$, $F_4$, and $F_5$. We have yet to cover any cross differences from the set $\{1,2,\ldots, x-1\}$ between groups $G_{i_1}$ and $G_{i_2}$ for any edge $(i_1,i_2) \in F_j$, $j \in \{2,\ldots,5\}$. So for each such edge, let $d_1=\{1,3,5, \ldots, x-2\}$ and  $d_2=x-1-d_1$. For each pair, $\{d_1,d_2\}$, $G[\Delta]$ produces a $2x$-cycle factor by Lemma \ref{Lemma2}. Thus we construct a set of $m=4(\frac{x-1}{2})$ factors in this way, which we will denote as $\cal{G}$$_m$ . Consequently, a total of $(\frac{x+5}{2}) + (\frac{x-5}{2}) + 4 (\frac{x-1}{2})= (3x-2)$ $2x$-cycle factors have been produced.

\

At this stage, we have addressed and covered all edges with cross differences from the set $\{1,2, \ldots, x-1\}$ from $G_{i_1}$ to $G_{i_2}$ for each edge $(i_1,i_2) \in F_1, F_2, F_3, F_4, F_5$. Thus we now address the cross difference zero. For a given $t$, let $K^t$ denote the subgraph of $K_{(x:6)}$ which consists of all edges with cross difference zero on vertices $(i,t)$ for $i= 0, \ldots, 5$. Then to cover all edges with cross difference zero we must decompose $K^t$ for each $t \in \{0,1\ldots, x-1\}$. We will refer to the set of vertices in $K^t$ as vertices in {\it level} $t$. 

\

If $x \equiv 1 \pmod{6}$, recall that in $\cal{D}$ we covered cross difference zero on the edges $\{(i,0),(i+1,0)\}$ for $i=0,2,4$ and $\{(i,t),(i+1,t)\}$ for $i=1,3,5$, and $t=1,2, \ldots, x-2$. In $\cal{E}$ we covered cross difference zero on $\{(i,0),(i+1,0)\}$ for $i=1,3,5$ and $\{(i,t),(i+1,t)\}$ for $i=0,2,4$, and $t=1,2, \ldots, x-2$. Thus we have covered cross difference zero from $F_1 \cup F_2$ on levels $t=0,1,2, \ldots, x-2$. Additionally, in $\cal{D}$ we covered cross difference zero on the edges $\{(i,t),(i+1,t)\}$ for $i=1,3,5$ and $t=x-1$, while in $\cal{E}$ we covered cross difference zero on $\{(i,t),(i+3,t)\}$ for $i=0,1,2$ and $t=x-1$. Thus we have covered cross difference zero from $F_1' \cup F_2'$ on level $t=x-1$. Now, place a copy of $F_3 \cup F_4$ on levels $t=0,1,2, \ldots, x-2$, and a copy of $F_3' \cup F_4'$ on level $t=x-1$. This produces a 6-cycle factor. Finally, place a copy of $F_5$ on levels $t=0,1,2, \ldots, x-2$, and a copy of $F_5'$ on level $t=x-1$, to obtain a 1-factor. 

\

If $x \equiv 5 \pmod{6}$, recall that in $\cal{D}$ we covered cross difference zero on the edges $\{(i,0),(i+1,0)\}$ for $i=0,2,4$ and $\{(i,t),(i+1,t)\}$ for $i=1,3,5$, and $t=1,2, \ldots, x-2$. In $\cal{E}$ we covered cross difference zero on $\{(i,0),(i+1,0)\}$ for $i=1,3,5$ and $\{(i,t),(i+1,t)\}$ for $i=0,2,4$, and $t=1,2, \ldots, x-2$. Thus we have covered cross difference zero from $F_1 \cup F_2$ on levels $t=0,1,2, \ldots, x-2$. Additionally, in $\cal{D}$ we covered cross difference zero on the edges $\{(i,t),(i+1,t)\}$ for $i=0,2,4$ and $t=x-1$, while in $\cal{E}$ we covered cross difference zero on $\{(i,t),(i+3,t)\}$ for $i=0,1,2$ and $t=x-1$. Thus we have covered cross difference zero from $F_1' \cup F_2'$ on level $t=x-1$. Now, place a copy of $F_3 \cup F_4$ on levels $t=0,1,2, \ldots, x-2$, and a copy of $F_3^* \cup F_4^*$ on level $t=x-1$, which produces a 6-cycle factor. Finally, place a copy of $F_5$ on levels $t=0,1,2, \ldots, x-2$, and a copy of $F_5^*$ on level $t=x-1$, to obtain the 1-factor. Thus in each case we have produced a 6-cycle factor, $\cal{H}$, and a 1-factor, $\cal{I}$. Furthermore, the edges with cross differences zero from $\cal{D} \cup \cal{E} \cup \cal{H} \cup \cal{I}$ decompose $K^t$.

\

Finally, $\cal{A}$$_k$ $\cup \cal{B}$$_k$  $\cup \cal{C} \cup \cal{D} \cup \cal{E} \cup \cal{F}$$_l$ $\cup \cal{G}$$_m$ $\cup \cal{H} \cup \cal{I}$ provides the desired solution to ${\hbox{HWP}}(6x;6,2x;1, \\ 3x-2)$.

\end{proof}

\subsection{$\gcd(6,N)=3$}

In the remainder of this paper, we focus on the case when $\gcd(6,N)=3$, we will prove there is a solution to $\hbox{HWP}(6xt;6,3x;1,3xt-2)$ for all odd $x \geq 3$.

\

\begin{lemma}
\label{18;6,9;1,7}
There exists a solution to ${\hbox{HWP}}(18;6,9;1,7)$.

\end{lemma}
\begin{proof} 

Let $G_0,G_1,G_2$ be the three parts of $K_{(6:3)}$ with $G_i$ having vertex set $\{(i,j): 0 \leq j \leq 5\} $.

\

Partition the set of points into 3 groups $G_0, G_1$ and $G_2$ of size 6, then the factors are given by:

\

${F_1}= \{(0,0), (1,1), (2,2), (2,0), (0,1), (1,2), (1,3), (2,4), (0,5), (0,0)\} \cup $ 

$\{(0,2), (2,1), (1,0), (1,5), (0,4), (2,3), (2,5), (1,4), (0,3), (0,2)\}$

\

${F_2}= \{(0,0), (2,5), (2,4), (2,0), (1,5), (1,4), (1,3), (0,2), (0,1), (0,0)\} \cup$ 

$\{(1,0), (1,1), (1,2), (2,3), (2,1), (2,2), (0,3), (0,4), (0,5), (1,0)\}$

\

${F_3}= \{(0,0), (0,4), (1,4), (2,4), (2,2), (0,2), (1,2), (1,0), (2,0), (0,0)\} \cup$ 

$\{(0,1), (0,3), (2,3), (1,3), (1,5), (0,5), (2,5), (2,1), (1,1), (0,1)\}$

\

${F_4}= \{(0,0), (1,0), (1,4), (1,2), (2,2), (2,3), (2,4), (0,4), (0,2), (0,0)\} \cup$ 

$\{(0,1), (2,1), (2,0), (2,5), (1,5), (1,1), (1,3), (0,3), (0,5), (0,1)\}$ 

\

${F_5}= \{(0,0), (1,2), (2,5), (0,4), (1,0), (2,3), (0,2), (1,4), (2,1), (0,0)\} \cup$ 

$\{(0,1), (1,3), (2,0), (0,5), (1,1), (2,4), (0,3), (1,5), (2,2), (0,1)\}$

\

${F_6}= \{(0,0), (1,3), (2,2), (0,4), (1,1), (2,0), (0,2), (1,5), (2,4), (0,0)\} \cup$ 

$\{(0,1), (1,4), (2,3), (0,5), (1,2), (2,1), (0,3), (1,0), (2,5), (0,1)\}$

\

${F_7}= \{(0,0), (1,5), (2,1), (0,4), (1,3), (2,5), (0,2), (1,1), (2,3), (0,0)\} \cup$ 

$\{(0,1), (1,0), (2,2), (0,5), (1,4), (2,0), (0,3), (1,2), (2,4), (0,1)\}$

\

${F_8}= \{(0,0), (1,4), (2,2), (2,5), (1,1), (0,3), (0,0)\} \cup$ 

$\{(0,1), (1,5), (2,3), (2,0), (1,2), (0,4), (0,1)\} \cup$ 

$\{(0,2), (1,0), (2,4), (2,1), (1,3), (0,5), (0,2)\}$

\

${F_9}= \{(0,0), (2,2)\} \cup$ $\{(0,1), (2,3)\} \cup$ $\{(0,2), (2,4)\} \cup$ $\{(0,3), (2,5)\} \cup$ 

$\{(0,4), (2,0)\} \cup$ $\{(0,5), (2,1)\} \cup$ $\{(1,0), (1,3)\} \cup$ $\{(1,1), (1,4)\} \cup$ $\{(1,2), (1,5)\}.$

\

Note that $F_1$ to $F_7$ are the seven 9-cycle factors of $K_{18}$, $F_8$ is the  6-cycle factor and $F_9$ is the 1-factor. Factors $F_1$ and $F_2$ are generated by the combination of 6-cycles of $K_6$ and $d_k=1$ for $k \in \{0,1,2\}$. Factors $F_3$ and $F_4$ are the result of combining triangles of $K_6$ and $d_k=0$ for $k \in \{0,1,2\}$. Finally, factors $F_5$ to $F_7$ arise from taking $d_k=2$ for $k \in \{0,1,2\}$, $d_k=3$ for $k \in \{0,1,2\}$, and $d_k=5$ for $k \in \{0,1,2\}$ respectively, then, by Lemma \ref{Lemma2}, $G[\Delta]$ consists of a $9$-cycle.

\end{proof}

\

\begin{lemma} 
\label{$gcd(3,N), t=1$}
There exists a solution to ${\hbox{HWP}}(6x;6,3x;1,3x-2)$ for all odd $x \geq 3$.

\end{lemma}
\begin{proof} 

When $x=3$, we know that there exists a ${\hbox{HWP}}(18;6,9;1,7)$ by Lemma \ref{18;6,9;1,7}.

Now suppose $x > 3$. Then by Theorem \ref{fac}, $K_{2x}$ can be decomposed into $(x-1)$ $C_x$-factors.

Give each point weight 3, so $K_{6x}$ can be considered as $K_{(3:2x)} \cup 2xK_3$. Then on one $C_x$ factor, construct four $C_{3x}$-factors of $K_{6x}$ as follows.

\begin{enumerate}

    \item Place $H_1$ on each $C_{(3:x)}$, where $H_1$ is given in Figure \ref{H_1}.
    \begin{figure}[h!]
\centering
\definecolor{mygray}{RGB}{70,80,70}

\begin{tikzpicture}[thick,
  every node/.style={circle, inner sep=0.1cm},
  fsnode/.style={fill=mygray},
  gsnode/.style={fill=mygray},
  hsnode/.style={fill=mygray},
  isnode/.style={fill=mygray},
  jsnode/.style={fill=mygray}
]

\begin{scope}[start chain=going below,node distance=7mm]
\foreach \i in {1,2,...,3}
  \node[fsnode,on chain] (f\i) {};
\end{scope}

\begin{scope}[xshift=1cm,yshift=0cm,start chain=going below,node distance=7mm]
\foreach \i in {1,2,...,3}
  \node[gsnode,on chain] (g\i) {};
\end{scope}

\begin{scope}[xshift=2cm,yshift=0cm,start chain=going below,node distance=7mm]
\foreach \i in {1,2,...,3}
  \node[hsnode,on chain] (h\i) {};
\end{scope}

\begin{scope}[xshift=3cm,yshift=0cm,start chain=going below,node distance=7mm]
\foreach \i in {1,2,...,3}
  \node[isnode,on chain] (i\i) {};
\end{scope}

\begin{scope}[xshift=6cm,yshift=0cm,start chain=going below,node distance=7mm]
\foreach \i in {1,2,...,3}
  \node[isnode,on chain] (j\i) {};
\end{scope}

\node[right=of i1] {$\cdots$};
\node[right=of i2] {$\cdots$};
\node[right=of i3] {$\cdots$};

\node [fit=(f1) (f3)] {};
\node [fit=(g1) (g3)] {};
\node [fit=(h1) (h3),label=above:$H_1$] {};
\node [fit=(i1) (i3)] {};
\node [fit=(j1) (j3)] {};

\draw (f1) edge[bend right] node [left] {}(f2);
\draw (f2) edge[bend right] node [left] {}(f3);
\draw (g1) edge[bend right] node [left] {}(g2);
\draw (g2) edge[bend right] node [left] {}(g3);
\draw (h1) edge[bend right] node [left] {}(h2);
\draw (h2) edge[bend right] node [left] {}(h3);
\draw (i1) edge[bend right] node [left] {}(i2);
\draw (i2) edge[bend right] node [left] {}(i3);
\draw (j1) edge[bend right] node [left] {}(j2);
\draw (j2) edge[bend right] node [left] {}(j3);
\draw (f3) -- (g1);
\draw (g3) -- (h1);
\draw (h3) -- (i1);
\draw (i3) -- (j1);
\draw (j3) edge[bend right] node [left] {}(f1);
\end{tikzpicture}
\caption{A $3x$-cycle from ${\hbox{HWP}}(6x;6,3x;1,3x-2)$ obtained from a $C_x$ of $K_{2x}$.}
\label{H_1}
\end{figure}
    \item Place $H_2$ on each $C_{(3:x)}$, where $H_2$ is given in Figure \ref{H_2}.
    \begin{figure}[h!]
\centering
\definecolor{mygray}{RGB}{70,80,70}

\begin{tikzpicture}[thick,
  every node/.style={circle, inner sep=0.1cm},
  fsnode/.style={fill=mygray},
  gsnode/.style={fill=mygray},
  hsnode/.style={fill=mygray},
  isnode/.style={fill=mygray},
  jsnode/.style={fill=mygray}
]

\begin{scope}[start chain=going below,node distance=7mm]
\foreach \i in {1,2,...,3}
  \node[fsnode,on chain] (f\i) {};
\end{scope}

\begin{scope}[xshift=1cm,yshift=0cm,start chain=going below,node distance=7mm]
\foreach \i in {1,2,...,3}
  \node[gsnode,on chain] (g\i) {};
\end{scope}

\begin{scope}[xshift=2cm,yshift=0cm,start chain=going below,node distance=7mm]
\foreach \i in {1,2,...,3}
  \node[hsnode,on chain] (h\i) {};
\end{scope}

\begin{scope}[xshift=3cm,yshift=0cm,start chain=going below,node distance=7mm]
\foreach \i in {1,2,...,3}
  \node[isnode,on chain] (i\i) {};
\end{scope}

\begin{scope}[xshift=6cm,yshift=0cm,start chain=going below,node distance=7mm]
\foreach \i in {1,2,...,3}
  \node[isnode,on chain] (j\i) {};
\end{scope}

\node[right=of i1] {$\cdots$};
\node[right=of i2] {$\cdots$};
\node[right=of i3] {$\cdots$};

\node [fit=(f1) (f3)] {};
\node [fit=(g1) (g3)] {};
\node [fit=(h1) (h3),label=above:$H_2$] {};
\node [fit=(i1) (i3)] {};
\node [fit=(j1) (j3)] {};

\draw (f1) edge[bend right] node [left] {}(f3);
\draw (g1) edge[bend right] node [left] {}(g3);
\draw (h1) edge[bend right] node [left] {}(h3);
\draw (i1) edge[bend right] node [left] {}(i3);
\draw (j1) edge[bend right] node [left] {}(j3);
\draw (f1) -- (g2);
\draw (f2) -- (g3);
\draw (g1) -- (h2);
\draw (g2) -- (h3);
\draw (h1) -- (i2);
\draw (h2) -- (i3);
\draw (i1) -- (j2);
\draw (i2) -- (j3);
\draw (j1) -- (f2);
\draw (j2) -- (f3);
\end{tikzpicture}
\caption{A $3x$-cycle from ${\hbox{HWP}}(6x;6,3x;1,3x-2)$ obtained from a $C_x$ of $K_{2x}$.}
\label{H_2}
\end{figure}

    \item \begin{enumerate}
    \item If $x \equiv 0 \ \hbox{or} \ 2 \pmod{3}$, then for each cycle $c=(i_0, i_1, \ldots, i_{x-1})$ in the $C_x$-factor, let $d_k=0$ for $k = 0,1, \ldots, x-2$ and let $d_{x-1}=2$. Then by Lemma \ref{Lemma2}, $G[\Delta]$ consists of a $3x$-cycle on $C_{(3:x)}$.
    \item If $x \equiv 1 \pmod{3}$, then for each cycle $c=(i_0, i_1, \ldots, i_{x-1})$ in the $C_x$-factor, let $d_k=0$ for $k = 0,1, \ldots, x-3$ and let $d_{x-2}=2$ and $d_{x-1}=2$. Then by Lemma \ref{Lemma2}, $G[\Delta]$ consists of a $3x$-cycle on $C_{(3:x)}$.
    \end{enumerate}
    \item \begin{enumerate}
    \item If $x \equiv 0 \ \hbox{or} \ 2 \pmod{3}$, then for each cycle $c=(i_0, i_1, \ldots, i_{x-1})$ in the $C_x$-factor, let $d_k=2$ for $k = 0,1, \ldots, x-2$ and let $d_{x-1}=0$. Then by Lemma \ref{Lemma2}, $G[\Delta]$ consists of a $3x$-cycle on $C_{(3:x)}$.
    \item If $x \equiv 1 \pmod{3}$, then for each cycle $c=(i_0, i_1, \ldots, i_{x-1})$ in the $C_x$-factor, let $d_k=2$ for $k = 0,1, \ldots, x-3$ and let $d_{x-2}=0$ and $d_{x-1}=0$. Then by Lemma \ref{Lemma2}, $G[\Delta]$ consists of a $3x$-cycle on $C_{(3:x)}$.
    \end{enumerate}
\end{enumerate}

On each of the remaining $(x-2)$ $C_x$ factors we will consider the following cases:

\begin{itemize}
    \item $x \equiv 0 \ \hbox{or} \ 2 \pmod{3}$.
    \begin{enumerate}
        \item For each cycle $c=(i_0, i_1, \ldots, i_{x-1})$ in the $C_x$-factor, let $d_k=0$ for $k = 0,1, \ldots, x-2$ and let $d_{x-1}=1$. Then by Lemma \ref{Lemma2}, $G[\Delta]$ consists of a $3x$-cycle on $C_{(3:x)}$.
        \item For each cycle $c=(i_0, i_1, \ldots, i_{x-1})$ in the $C_x$-factor, let $d_k=1$ for $k = 0,1, \ldots, x-2$ and let $d_{x-1}=0$. Then by Lemma \ref{Lemma2}, $G[\Delta]$ consists of a $3x$-cycle on $C_{(3:x)}$.
        \item For each cycle $c=(i_0, i_1, \ldots, i_{x-1})$ in the $C_x$-factor, let $d_k=2$ for $k = 0,1, \ldots, x-1$. Then by Lemma \ref{Lemma2}, $G[\Delta]$ consists of a $3x$-cycle on $C_{(3:x)}$.
   \end{enumerate}
    \item $x \equiv 1 \pmod{3}$.

    \begin{enumerate}
        \item For each cycle $c=(i_0, i_1, \ldots, i_{x-1})$ in the $C_x$-factor, let $d_k=0$ for $k = 0,1, \ldots, x-4$ and let $d_{x-3}=2$, $d_{x-2}=1$, and $d_{x-1}=1$. Then by Lemma \ref{Lemma2}, $G[\Delta]$ consists of a $3x$-cycle on $C_{(3:x)}$.
        \item For each cycle $c=(i_0, i_1, \ldots, i_{x-1})$ in the $C_x$-factor, let $d_k=1$ for $k = 0,1, \ldots, x-4$ and let $d_{x-3}=1$, $d_{x-2}=2$, and $d_{x-1}=0$. Then by Lemma \ref{Lemma2}, $G[\Delta]$ consists of a $3x$-cycle on $C_{(3:x)}$.
        \item For each cycle $c=(i_0, i_1, \ldots, i_{x-1})$ in the $C_x$-factor, let $d_k=2$ for $k = 0,1, \ldots, x-4$ and let $d_{x-3}=0$, $d_{x-2}=0$, and $d_{x-1}=2$. Then by Lemma \ref{Lemma2}, $G[\Delta]$ consists of a $3x$-cycle on $C_{(3:x)}$.
   \end{enumerate}
        
\end{itemize}

Thus from the construction above, we have obtained a total of $(3x-2)$- $3x$ cycle factors. Now on the 1-factor of $K_{2x}$, we let $d_0=1$ and $d_1=1$. Then $d=2$, so by Lemma \ref{Lemma2}, $G[\Delta]$ consists of a 6-cycle factor. Also on the 1-factor of $K_{2x}$, let $d_0=0$ to obtain a 1-factor.

\end{proof}

\

\begin{theorem} 
\label{gcd=3}
There exists a ${\hbox{HWP}}(6xt;6,3x;1,3xt-2)$ for all odd $x > 3$.
\end{theorem}
\begin{proof} 
The case of $t=1$ is addressed in Lemma \ref{$gcd(3,N), t=1$}.

\

If $t=2$, we have the case ${\hbox{HWP}}(12x;6,3x;1,6x-2)$. 

\

When $x \geq 5$, we know $K_{(2:x)}$ has a resolvable $C_x$ factorization, by Theorem \ref{equip}, for all $x \geq5$. Notice that by Theorem \ref{fac}, $K_{2x}$ can be decomposed into $(2x-1)$ $C_x$-factors, $F_1, F_2, \ldots, F_{2x-1}$. Give each point of $K_{(2:x)}$ weight 3. Then on the $F_1$ factor, on every $x$-cycle, take cycles of each of the four types given in Lemma \ref{$gcd(3,N), t=1$} to obtain four ${3x}$-cycle factors. Then, on the remaining $F_2, \ldots, F_{2x-1}$, on every $x$-cycle, take cycles of each of the three types to create three $3x$-cycles. This produces $(6x-2)$- $3x$ cycle factors. Now on the 1-factor of $K_{2x}$, we let $d_0=0$ and $d_1=1$. Then $d=1$, so by Lemma \ref{Lemma2}, $G[\Delta]$ consists of a 6-cycle factor. As for the 1 factors, we let $d_0=1$, then by Lemma \ref{Lemma2}, $G[\Delta]$ consists of 1 factors.

\

Now suppose $t \geq 3$. $K_{6xt}$  can be considered as $K_{(6x:t)} \cup tK_{6x}$. By Theorem \ref{equip}, there exists a decomposition of $K_{(6x:t)}$  into $C_{3x}$ factors. There also exists a solution to ${\hbox{HWP}}(6x;6,3x;1,3x-2)$ by Lemma \ref{$gcd(3,N), t=1$}.

\end{proof}

\

We now proceed to prove our main Theorem.

\begin{theorem}
\label{Main}
    Let $x$, $t$, and $p$ be integers such that $t \geq 1$, $x \geq 2$ and $x \neq 6p$. A solution to $\hbox{HWP}(v;6,N;1,\beta)$, where $v=6xt$, $N=2x$, and $\beta=3xt-2$, exists if and only if $6|v$ and $N|v$, except possibly when $\gcd(6,N)=1$, or $\gcd(6,N)=2$ and $x=3$.
\end{theorem}

\begin{proof}
    If a solution exists, then by Theorem \ref{Ncond}, $6|v$ and $N|v$. Let $l=\lcm(6,N)$. If $6|v$ and $N|v$, then by Theorem \ref{Burgess}, there exists a solution to $\hbox{HWP}(v;6,N;1,\beta)$ except possibly when

    \begin{itemize}
    \item $\gcd(6,N) \in \{1,2\}$;
    \item 4 does not divide $v/l$;
    \item $v/4l \in \{1, 2\}$;
    \item $v = 16l$ and $\gcd(6, N)$ is odd;
    \item $v = 24l$ and $\gcd(6, N) = 3$.
\end{itemize}

Now, if $6|v$ and $N|v$, then $\gcd(6,N) \in \{1,2,3,6\}$. If $\gcd(6,N)=6$, then this case was solved in \cite{BURDAN}. If $\gcd(6,N)=3$, then $N=3x$ with $x$ odd. There exists a solution to $\hbox{HWP}(6xt;6,3x;1,3xt-2)$ for all odd $x \geq 3$ in Theorem \ref{gcd=3}. Notice that $l=\lcm(6,3x)=6x$, so $\frac{v}{l}=t$. This result addresses the following open cases from Theorem \ref{Burgess}: when $4\not | t$, when $t=4 \ \hbox{or} \ 8$, when $v=96x$, and when $v=144x$ for odd $x > 3$. Thus the case of $x=3$ is still open. If $\gcd(6,N)=2$, then $N=2x$. The existence of a solution to $\hbox{HWP}(6xt;6,2x;1,3xt-2)$ with $x\equiv 2 \ \hbox{or} \ 4 \pmod{6}$ is given in Theorem \ref{1-gcd=2}. The existence of a solution to $\hbox{HWP}(6x;6,2x;1,3x-2)$ with $x\equiv 1 \ \hbox{or} \ 5 \pmod{6}$ is given in Theorem \ref{2-gcd=2}. These results produce solutions for the open case of Theorem \ref{Burgess} when $\gcd(6,N)=2$.
\end{proof}

\

\begin{section}{Conclusions and Future Work}

The results given in this paper almost settle the existence problem for $\hbox{HWP}(v;6,N;1,\beta)$. We are esentially left with the difficult case of $\gcd(M,N)=1$ . It seems a completely different technique should be used for a solution, if it exists. We also believe that it might be possible to generalize the idea of $MRSM$s in order to extend these results and use the algorithm developed in this paper to build these matrices using a computing system such as Magma for $M >6$.

\end{section}

\end{document}